\documentclass[11pt]{amsart}

\usepackage{amscd}
\usepackage{eucal}
\usepackage{enumerate}
\usepackage{amsfonts}
\usepackage{amssymb}
\usepackage{amsmath}
\usepackage{amsthm}
\usepackage[all]{xy}
\usepackage{array}
\usepackage{tabularx}
\usepackage{colortbl}

\newcommand{\coker}{\operatorname{coker}}

\newcommand{\coend}{\! \int \! }
\newcommand{\cs}{{ \bbc^{\times} } }
\newcommand{\zs}{\bbz^\times}
\newcommand{\rs}{\bbr^\times}

\newcommand{\csi}[1]{{(\bbc^\times)^{#1}_i}}

\newcommand{\cone}{\operatorname{Cone}}

\newcommand{\redcxp}{\calc^*(\calz_0(S^{p,p}))}

\newcommand{\spec}{\operatorname{Spec}}
\newcommand{\bor}{\text{bor}}
\newcommand{\uhom}[2]{\underline{Hom}(#1,#2)}

\newcommand{\sig}{\aleph}
\newcommand{\sigtor}{\aleph_\text{tor}}

\newcommand{\brAp}{A(p)_{\mathcal{B}r}}       
\newcommand{\brApbor}{A(p)^\bor_{\mathcal{B}r}}       
\newcommand{\brAq}{A(q)_{\mathcal{B}r}}       
\newcommand{\brZp}{\bbz(p)_{\mathcal{B}r}}    

\newcommand{\br}[2]{{#1}(#2)_{\mathcal{B}r}}  

\newcommand{\dAp}{A(p)_{\mathcal{D}/\bbr}}    
\newcommand{\dApbor}{A(p)_{\mathcal{D}/\bbr}^\bor}    
\newcommand{\dAq}{A(q)_{\mathcal{D}/{\bbr}}}    
\newcommand{\dZp}{\bbz(p)_{\mathcal{D}/{\bbr}}} 
\newcommand{\dZpc}{\bbz(p)_{\mathcal{D}/{\bbc}}} 

\newcommand{\de}[2]{{#1}(#2)_{\mathcal{D}/\bbr}} 

\newcommand{\ab}{Ab}
\newcommand{\rav}{\cala n_{\! /\bbr}}      
\newcommand{\cav}{\cala n_{\! /\bbc}}      
\newcommand{\gman}{\frS\text{-}\calm an}    
\newcommand{\gtop}{\frS\text{-}\calt op}
\newcommand{\gfin}{\frS\text{-Fin}}  
\newcommand{\srv}{Sm_{\! /\bbr}}        
\newcommand{\scv}{Sm_{/\bbc}}        

\newcommand{\rlbc}[1]{PW^{\nabla}(#1)}

\newcommand{\dcrA}[3]{H^{#1}_{\cald/\bbr}(#3;A(#2))} 

\newcommand{\dcrAbor}[3]{H^{#1}_{\cald/\bbr,\bor}(#3;A(#2))} 

\newcommand{\dccA}[3]{H^{#1}_{\cald/\bbc}(#3;A(#2))} 

\newcommand{\dcr}[3]{H^{#1}_{\cald/\bbr}(#3;\bbz(#2))} 

\newcommand{\dcrr}[3]{H^{#1}_{\cald/\bbr}(#3;\bbr(#2))} 

\newcommand{\dcc}[3]{H^{#1}_{\cald/\bbc}(#3;\bbz(#2))} 

\newcommand{\bcz}[3]{H_{\text{Br}}^{#1,#2}(#3,\uZ)}     

\newcommand{\bca}[3]{H_{\text{Br}}^{#1,#2}(#3,\uA)}     

\newcommand{\bcabor}[3]{H_{\text{bor}}^{#1,#2}(#3,\uA)}     

\newcommand{\bcm}[4]{H_{\text{Br}}^{#1,#2}(#3,#4)}     

\newcommand{\scm}[3]{H_{\text{sing}}^{#1}(#2;#3)}     

\newcommand{\cov}[1]{{Cov}(#1)}
\newcommand{\eqsec}[1]{\Gamma_{\symg{2}}()} 

\newcommand{\sst}{\scriptstyle}
\newcommand{\vD}{\varDelta}

\newcommand{\frS}{\mathfrak{S}}

\newcommand{\field}[1]{\ensuremath{\mathbb{#1}}}
\newcommand{\bba}{\field{A}}
\newcommand{\bbc}{\field{C}}
\newcommand{\bbf}{\field{F}}
\newcommand{\bbh}{\field{H}}

\newcommand{\bbp}{\field{P}}
\newcommand{\bbq}{\field{Q}}
\newcommand{\bbr}{\field{R}}

\newcommand{\bbz}{\field{Z}}

\newcommand{\uA}{\underline{A}}

\newcommand{\uZ}{\underline{\bbz}}
\newcommand{\uM}{\underline{M}}

\newcommand{\cala}{\mathcal{A}}
\newcommand{\calb}{\mathcal{B}}
\newcommand{\calc}{\mathcal{C}}
\newcommand{\cald}{\mathcal{D}}
\newcommand{\cale}{\mathcal{E}}
\newcommand{\calf}{\mathcal{F}}

\newcommand{\calk}{\mathcal{K}}
\newcommand{\call}{\mathcal{L}}
\newcommand{\calm}{\mathcal{M}}
\newcommand{\caln}{\mathcal{N}}
\newcommand{\calo}{\mathcal{O}}
\newcommand{\calp}{\mathcal{P}}

\newcommand{\calr}{\mathcal{R}}

\newcommand{\calt}{\mathcal{T}}
\newcommand{\calu}{\mathcal{U}}
\newcommand{\calz}{\mathcal{Z}}

\newcommand{\bone}{\mathbf{1}}

\newcommand{\mbc}{\mathbf{c}}

\newcommand{\mbh}{\mathbf{h}}

\newcommand{\mbq}{\mathbf{q}}

\newcommand{\mbt}{\mathbf{t}}

\newcommand{\ve}{\varepsilon}

\newcommand{\la}{\langle}
\newcommand{\ra}{\rangle}

\newcommand{\symg}[1]{\mathfrak{S}_{#1}}

\newcommand{\equdef}{:=}

\renewcommand{\hom}[3]{\operatorname{Hom}_{#1}(#2, #3)}

\newtheorem{theorem}{Theorem}[section]
\newtheorem{lemma}[theorem]{Lemma}
\newtheorem{proposition}[theorem]{Proposition}
\newtheorem{corollary}[theorem]{Corollary}

\theoremstyle{definition}

\newtheorem{definition}[theorem]{Definition}
\newtheorem{example}[theorem]{Example}

\newtheorem{properties}[theorem]{Properties}

\newtheorem{remark}[theorem]{Remark}

\theoremstyle{remark}

\newtheorem*{ack}{\bf Acknowledgements}

 \date{May 2008}
 \title{Integral Deligne Cohomology for Real Varieties}
 \author[dos Santos]{Pedro F. dos Santos}
 \address{Departamento de Matem\'atica, Instituto Superior
 T\'ecnico, Portugal}
 \email{pedro.f.santos@math.ist.utl.pt}
 \author[Lima-Filho]{Paulo Lima-Filho}
 \address{Department of
 Mathematics, Texas A{\&}M University, USA}
 \email{plfilho@math.tamu.edu}
 \thanks{The first author was supported in part by  FCT (Portugal)
 through program POCTI}

\begin{document}

\begin{abstract}
We develop an {\bf integra}l version of Deligne cohomology  for smooth proper {\bf real} varieties $Y$. For this purpose the role played by singular cohomology in the complex case has to be replaced by \emph{ordinary bigraded $\frS$-equivariant cohomology}, where $\frS:= Gal(\bbc/\bbr)$. This is the $\frS$-equivariant counterpart of singular cohomology; cf. \cite{may-bull}. We establish the basic properties of the theory and give a geometric interpretation for the groups in dimension $2$ in weights $1$ and $2$.

\end{abstract}

\maketitle

\tableofcontents

\vfill\eject

\section{Introduction}

Integral Deligne cohomology $\dcc{n}{p}{Y}$ for a complex variety $Y$ has been widely studied in the literature. When $Y$ is smooth and proper, it is defined as the hypercohomology group  $\bbh^n(Y,\dZpc)$ of the complex
\begin{equation}
\label{eq:dc}
\dZpc\ : \ 0 \to \bbz(p) \to \calo \xrightarrow{d} \Omega^1 \xrightarrow{d} \to \cdots \to \Omega^{p-1},
\end{equation}
where $\bbz(p)$ is the constant sheaf $(2\pi i)^p\bbz$ and $\Omega^j$ is the sheaf of holomorphic $j$-forms. More generally, one can extend this setting to arbitrary complex varieties using suitable compactifications and simplicial resolutions, along with forms with logarithmic poles, and the resulting theory is called Deligne-Be{\u\i}linson cohomology; see \cite{bei-higher}.  An excellent account can be found in \cite{esnault}. The theory is complemented by a homological counterpart and together they are shown to satisfy Bloch-Ogus' formalism \cite{blo&ogu}; cf. \cite{gil-dh} and \cite{jan-dh}.

In this paper we develop an integral version of Deligne cohomology for smooth proper {\bf real} varieties. For this purpose the role played by singular cohomology $H^n_\text{sing}(Y,\bbz(p))$ in the complex case has to be replaced by \emph{ordinary bigraded $\frS$-equivariant cohomology} $\bcz{n}{p}{Y(\bbc)}$, where $\frS:= Gal(\bbc/\bbr)$. This is the $\frS$-equivariant counterpart of singular cohomology; cf. \cite{may-bull}.

We must emphasize that the ordinary equivariant cohomology of a $\frS$-space $X$ is not obtained as the singular cohomology of the Borel construction $E\frS\times_\frS X$. In fact, the ``Borel version'' of equivariant cohomology is just  $\bcz{n}{0}{E\frS \times X}$ and, more generally,
$\bcz{n}{p}{E\frS \times X} \cong H^n(E\frS\times_\frS X;\bbz(p))$, where the latter denotes cohomology with twisted coefficients $\bbz(p)$. The bigraded version stems from the $RO(\frS)$-graded equivariant cohomology theories developed by J. Peter May et al. in \cite{may-bull}, \cite{may-SLN1213} and \cite{may-CBMS}. With weight $p=0$ ordinary equivariant cohomology was developed in \cite{bredon} and hence, for simplicity, we call it \emph{Bredon cohomology} even with non-zero weights, thus explaining the notation $\bcz{n}{p}{X}$.

From a motivic standpoint, this difference can be expressed by saying that ordinary equivariant cohomology is to its Borel version as motivic cohomology is to its \'etale counterpart (\'etale-motivic cohomology; cf. \cite{maz&wei}). In fact, this setting is much more that a mere analogy, once one observes that the topological realization of the $\bba^1$-homotopy category of schemes over $\bbr$ lands in the $\frS$-equivariant homotopy category. This realization carries motivic cohomology to ordinary bigraded equivariant cohomology, and carries \'etale-motivic cohomology to Borel equivariant cohomology (with twisted coefficients). See \cite{mor&voe-A1} and \cite{dug&isa-A1} for details.

Let $\rav$ denote the category of \emph{real holomorphic manifolds}, whose objects are pairs  $(M,\sigma)$ consisting of a holomorphic manifold $M$ together with an anti-holomorphic involution $\sigma$, and whose morphisms are holomorphic maps commuting with the involutions. We consider those objects as having an action of $\frS$ and give $\rav$ the structure of a site using equivariant open covers. In order to introduce Deligne cohomology for proper real holomorphic manifolds, we construct a real version $\dZp$ of the Deligne complex~\eqref{eq:dc} on the site $\rav$ and define
\begin{equation}
\label{eq:def-dc}
\dcr{n}{p}{M}:= \check{\bbh}^n(M_{eq};{\dZp}).
\end{equation}

The  construction of $\dZp$ involves a replacement of the constant sheaf $\bbz(p)$ by a complex $\brZp$ which computes ordinary equivariant cohomology; cf. \cite{yang-pre}. The key technical construction is a morphism of complexes
$$
\tau_p \colon \brZp \longrightarrow \cale^*,
$$
where $\cale^*$ is the complex of sheaves on $\rav$ such that $\cale^j(X)$ consists of those smooth complex-valued differential $j$-forms invariant under the simultaneous action of $\frS$ on $X\in \rav$ and $\bbc.$ Using this complex we define
$$
\dZp := \cone\left( \brZp \oplus F^p\cale^* \xrightarrow{\tau_p\ -\ \imath}
\cale^*\right)[-1]
$$
where $F^p\cale^*$ is the $p$-th piece of the Hodge filtration.

Amongst the basic properties of the theory are the evident long exact sequences and the fact that it recovers the usual theory for complex varieties. More precisely,
 given a complex projective variety $Y_\bbc$, let and $Y_{\! \bbc/\bbr}$ denote the associated real variety obtained by restriction of scalars.
 \smallskip

{\noindent{\bf Proposition \ref{prop:change}.}\ {\it Let $Y_\bbc$ be a proper complex holomorphic  manifold and let $X$ be a proper smooth real algebraic variety.
\begin{enumerate}[i.]
\item One has  natural isomorphisms
$$
\dcrA{i}{p}{Y_{\bbc/\bbr}} \cong \dccA{i}{p}{Y_\bbc},
$$
where the latter denotes the usual Deligne cohomology of the complex
manifold $Y_\bbc$. (See Remark \ref{rem:notation}.)
\item If $X_\bbc$ is the complex variety obtained from $X$ by base
  extension, the corresponding map of real varieties $X_{\bbc/\bbr}
  \to X$ induces natural homomorphisms
$$
\dcrA{i}{p}{X} \to  \dccA{i}{p}{X_\bbc}^\frS,
$$
where the latter denotes the invariants of the Deligne cohomology of
the complex variety $X_\bbc$.
~
\item  One has a long exact sequence
{\small \begin{multline}
\cdots \to H_\text{sing}^{j-1}(X(\bbc),\bbc)^\frS \to \dcrA{j}{p}{X}\xrightarrow{\nu} \\
\to \bca{j}{p}{X(\bbc)} \oplus \left\{ F^pH_\text{sing}^j(X(\bbc);\bbc)
\right\}^\frS \xrightarrow{\ \ \jmath\circ \varphi - \imath \ \ }
H_\text{sing}^j(X(\bbc);\bbc)^\frS \to \cdots
\end{multline}
}
\end{enumerate}
}
}

\smallskip

\noindent Here $X(\bbc)$ denotes the set of complex points of $X$ with the analytic topology and  $\left\{ F^pH^j(X(\bbc);\bbc) \right\}^\frS $ denotes the invariants of the $p$-th level of the Hodge filtration on singular cohomology, under the simultaneous action of $\frS$ on $X(\bbc)$ and on the coefficients $\bbc$ of the cohomology.
\smallskip

\newcolumntype{H}{>{\columncolor[gray]{0.8}[0.9\tabcolsep]}c}
\newcolumntype{P}{>{\columncolor[gray]{0.86}[0.9\tabcolsep]}c}
\newcolumntype{F}{%
>{\color{blue}
\columncolor{white}[.6\tabcolsep]}c|}
\newcolumntype{U}{%
>{\color{white}
\columncolor{white}[.5\tabcolsep]}c}
\newcolumntype{V}{%
>{\color{magenta}
\columncolor[rgb]{0.96,0.96,0.9}[0.9\tabcolsep]}c}

\begin{table}[!h]
\begin{tabular}{l!{\vrule}cccccccccccc}
%
\multicolumn{1}{F}{8} &
%
\multicolumn{5}{V}{} &
%
\multicolumn{1}{V}{$0$} &
\multicolumn{1}{H}{$\bbr/\bbz(8)$} &
\multicolumn{1}{P}{$\zs$} &
\multicolumn{1}{V}{$0$} &
\multicolumn{1}{P}{$\zs$} &
\multicolumn{1}{V}{$0$}  \\
%
%
\multicolumn{1}{F}{7} &
%
\multicolumn{5}{V}{} &
%
\multicolumn{1}{V}{$0$} &
\multicolumn{1}{P}{$\rs$} &
\multicolumn{1}{V}{$0$} &
\multicolumn{1}{P}{$\zs$} &
\multicolumn{1}{V}{$0$} &
\multicolumn{1}{P}{$\zs$}  \\
%
%
\multicolumn{1}{F}{6} &
%
\multicolumn{5}{V}{} &
%
\multicolumn{1}{V}{$0$} &
\multicolumn{1}{H}{$\bbr/\bbz(6)$} &
\multicolumn{1}{P}{$\zs$} &
\multicolumn{1}{V}{$0$} &
\multicolumn{1}{P}{$\zs$} &
\multicolumn{1}{V}{$0$}  \\
%
%
\multicolumn{1}{F}{5} &
%
\multicolumn{5}{V}{} &
%
\multicolumn{1}{V}{$0$} &
\multicolumn{1}{P}{$\rs$} &
\multicolumn{1}{V}{$0$} &
\multicolumn{1}{P}{$\zs$} &
\multicolumn{1}{V}{$0$} &
\multicolumn{1}{P}{$\zs$}  \\
%
%
\multicolumn{1}{F}{4} &
\multicolumn{5}{V}{} &
\multicolumn{1}{V}{$0$} &
\multicolumn{1}{H}{$\bbr/\bbz(4)$ }&
\multicolumn{1}{P}{$\zs$} &
\multicolumn{1}{V}{$0$} &
\multicolumn{1}{P}{$\zs$} &
\multicolumn{1}{V}{$0$}  \\
%
%
\multicolumn{1}{F}{3} &
\multicolumn{5}{V}{} &
\multicolumn{1}{V}{$0$} &
\multicolumn{1}{P}{$\rs$} &
\multicolumn{1}{V}{$0$} &
\multicolumn{1}{P}{$\zs$} &
\multicolumn{1}{V}{$0$} &
\multicolumn{1}{V}{$0$}  \\
\multicolumn{1}{F}{2} &
\multicolumn{5}{V}{} &
\multicolumn{1}{V}{$0$} &
\multicolumn{1}{H}{$\bbr/\bbz(2)$} &
\multicolumn{1}{P}{$\zs$} &
\multicolumn{1}{V}{$0$} &
\multicolumn{1}{V}{$0$} &
\multicolumn{1}{V}{$0$}  \\
%
%
\multicolumn{1}{F}{1} &
\multicolumn{5}{V}{} &
\multicolumn{1}{V}{$0$} &
\multicolumn{1}{P}{$\rs$} &
\multicolumn{1}{V}{$0$} &
\multicolumn{1}{V}{$0$} &
\multicolumn{1}{V}{$0$} &
\multicolumn{1}{V}{$0$}   \\
%
%
\multicolumn{1}{F}{0} &
\multicolumn{5}{V}{} &
\multicolumn{1}{H}{$\bbz(0)$} &
\multicolumn{1}{V}{$0$} &
\multicolumn{1}{V}{$0$} &
\multicolumn{1}{V}{$0$} &
\multicolumn{1}{V}{$0$} &
\multicolumn{1}{V}{$0$}   \\
\hline
%
%
\multicolumn{1}{F}{-1} &
\multicolumn{5}{V}{} &
\multicolumn{1}{V}{$0$} &
\multicolumn{5}{V}{} &  \\
%
%
\multicolumn{1}{F}{-2} &
\multicolumn{5}{V}{} &
\multicolumn{1}{H}{$2\bbz(-2)$} &
\multicolumn{5}{V}{} & \\
%
%
\multicolumn{1}{F}{-3} &
\multicolumn{5}{V}{} &
\multicolumn{1}{P}{$\zs$} &
\multicolumn{5}{V}{} \\
%
%
\multicolumn{1}{F}{-4} &
\multicolumn{4}{V}{} &
\multicolumn{1}{P}{$\zs$} &
\multicolumn{1}{H}{$2\bbz(-4)$} &
\multicolumn{5}{V}{}   \\
%
%
\multicolumn{1}{F}{-5} &
\multicolumn{3}{V}{} &
\multicolumn{1}{P}{$\zs$} &
\multicolumn{1}{V}{$0$} &
\multicolumn{1}{P}{$\zs$} &
\multicolumn{5}{V}{} \\
%
\multicolumn{1}{F}{-6} &
\multicolumn{2}{V}{} &
\multicolumn{1}{P}{$\zs$} &
\multicolumn{1}{V}{$0$} &
\multicolumn{1}{P}{$\zs$} &
\multicolumn{1}{H}{$2\bbz(-6)$} &
\multicolumn{5}{V}{}  \\
%
%
\multicolumn{1}{F}{-7} &
\multicolumn{1}{V}{} &
\multicolumn{1}{P}{$\zs$} &
\multicolumn{1}{V}{$0$} &
\multicolumn{1}{P}{$\zs$} &
\multicolumn{1}{V}{$0$} &
\multicolumn{1}{P}{$\zs$} &
\multicolumn{5}{V}{}  \\
%
\multicolumn{1}{F}{-8} &
\multicolumn{1}{P}{$\zs$} &
\multicolumn{1}{V}{$0$} &
\multicolumn{1}{P}{$\zs$} &
\multicolumn{1}{V}{$0$} &
\multicolumn{1}{P}{$\zs$} &
\multicolumn{1}{H}{$2\bbz(-8)$} &
\multicolumn{5}{V}{}  \\ \hline
%
%
%
\multicolumn{1}{F}{} &
{\color{blue} -5} &
{\color{blue} -4} &
{\color{blue} -3} &
{\color{blue} -2} &
{\color{blue} -1} &
{\color{blue} 0 } &
{\color{blue} 1 } &
{\color{blue} 2} &
{\color{blue} 3} &
{\color{blue} 4} &
{\color{blue} 5}
\end{tabular}
\label{table:dc-pt}
\setlength{\abovecaptionskip}{22pt}   
\setlength{\belowcaptionskip}{6pt}   
\caption{Cohomology of a point}
\end{table}


We define Deligne cohomology with negative weights $p<0$ to coincide with ordinary equivariant cohomology. Using this convention, we give $\oplus_{n, p} \dcr{n}{p}{X}$ a bigraded ring structure compatible with various functors in the theory and having many computational properties. For example, the bigraded group structure of the cohomology of a point $X = Spec(\bbr)$ is displayed in Table \ref{table:dc-pt} above, and its ring structure is displayed in Table \ref{Table2}, subsection~\ref{subsec:prodII}.

\bigskip

Using this product structure we obtain obtain formulae for the Deligne cohomology ring of some relevant examples, such as:
\smallskip

\noindent{\bf Corollary \ref{cor:pbf}.}\ {\it
Under the hypothesis of Proposition \ref{prop:projbf1}, one has an
isomorphism of bigraded rings
$$
\dcrA{*}{*}{X}[T]/\la T^{p+1} \ra \cong \dcrA{*}{*}{X\times
\bbp^p}
$$
when  $T$ is given the bigrading $(2,1).$
}
\smallskip

\noindent A more general projective bundle formula together with a theory of characteristic classes appear in a forthcoming paper.
\bigskip

The cases of weights $p=1$ and $p=2$ have interesting geometric interpretations.
We first show that $\de{\bbz}{1}$ is quasi-isomorphic to $\calo^\times [-1]$, cf. Corollary \ref{cor:qi}, and derive as a consequence an \emph{exponential sequence} relating the cohomology of $\br{\bbz}{1}$,  $\calo$ and
$\calo^\times$:
\smallskip

\noindent{\bf Corollary \ref{cor:exp-seq}.}\ {\it
Let $X$ be a smooth proper real algebraic variety. Then there is a
long exact sequence
{\footnotesize
$$
\to
H_{\text{Br}}^{*,1}(X(\bbc);\uZ)\xrightarrow{\vartheta}H^*(X;\calo_X)
\xrightarrow{\exp} H^*(X;\calo_X^*) \to H_{\text{Br}}^{*+1,1}(X(\bbc);\uZ)\to
$$
}
where $\vartheta$ denotes the composite
$$
H_{\text{Br}}^{*,1}(X(\bbc);\uZ)\xrightarrow{\tau_1}
H^*(X(\bbc);\bbc)^\mathfrak{G}\twoheadrightarrow
H^{*}(X;\calo_X)=H_{\overline{\partial}}^{*,1}(X(\bbc))^\mathfrak{G},
$$
and the latter denotes the invariants of the Dolbeault cohomology
of the complex manifold $X(\bbc)$.
}
\smallskip

When $X$ is a curve, the exponential sequence then gives:
\smallskip

\noindent{\bf Proposition \ref{prop:pic}.}\ {\it
Let $X$ be an irreducible, smooth, projective curve over $\bbr$, of genus
$g$.  Then
$
Pic(X) \cong Pic_0(X_\bbc)^\frS  \times \bcz{2}{1}{X(\bbc)}.
$
}
\smallskip

\noindent As a corollary to this proposition we provide a new proof of \emph{Weichold's Theorem} -- a classical result in real algebraic geometry which determines the Picard group of a real algebraic curve.
\smallskip

In order to give a geometric interpretation of $\dcr{2}{2}{X}$ for a real projective
variety $X$ we first provide in Proposition \ref{prop:h22br} alternative interpretations of the Bredon cohomology group $\bcz{2}{2}{Y}$ for any $\frS$-manifold $Y.$
Using Atiyah's terminology in \cite{ati-real}, a  \emph{Real vector bundle} $(E,\tau)$ on a
$\frS$-manifold $(Y,\sigma)$ consists of a complex vector bundle $E$
on $Y$ together with an isomorphism $\tau \colon \overline{\sigma^* E}
\to E$ satisfying $\tau \circ {\overline{\sigma^*\tau}} = Id$.
Now, consider the set $\call_2(Y)$ of equivalence classes of pairs
$\la L,\mbq \ra$ satisfying:
\begin{enumerate}[{\bf p1)}]
\item $L$ is a (smooth) complex line bundle on $Y$;
\item $\mbq \colon L\otimes \overline{\sigma^*L} \to \bone_Y$ is
  an isomorphism of Real line bundles, where $L\otimes
  \overline{\sigma^*L}$ carries the tautological Real line bundle structure.
\end{enumerate}
It follows that $\call_2(Y)$ becomes a group under the tensor product of line bundles and we show that this group is naturally isomorphic to $\bcz{2}{2}{Y}$ .

Now, let $S=\pi_0(Y^\frS)$ denote the set of connected components of
the fixed point set $Y^\frS$, and identify $H^0(Y^\frS;\bbz^\times) \equiv
(\bbz^\times)^S$. Given $\la L, \mbq
\ra \in \bcz{2}{2}{Y}$, the restriction of $\mbq$ to $L|_{Y^\frS}$\
becomes a non-degenerate hermitian pairing, and hence it has a
well-defined signature $\sig_{\la L, \mbq \ra} \in
(\bbz^\times)^S$.  We call  $$\sig \colon \bcz{2}{2}{Y} \to ({\bbz^\times})^S$$
the \emph{equivariant signature map} of $Y$
and the image $\sigtor(Y)\subseteq (\bbz^\times)^S$  of the torsion subgroup
$\bcz{2}{2}{Y}_\text{tor}$ under $\sig$  is called the
\emph{equivariant signature group} of $Y.$  In the case where
$Y=X(\bbc)$ for a real algebraic variety $X$ with $S =\pi_0( X(\bbr)
)$, we denote the equivariant signature group of $X(\bbc)$ simply by
$\sigtor(X)$.  For example, if $X$ is a real algebraic curve, then $\sigtor(X)$ is the Brauer group of $X$; see Section \ref{sec:exp}.

Given a proper real variety $X$, let $\rlbc{X}$ denote the set of
isomorphism classes $\la L,\nabla,\mbq\ra$ of triples
where $L$ is a holomorphic line bundle over $X(\bbc)$, $\nabla$ is
a holomorphic connection on $L$ and $\mbq \colon L\otimes
\overline{\sigma^*L}\to \bone $ is a holomorphic isomorphism of
Real line bundles satisfying the following properties:

\begin{enumerate}
\item The restriction of ${\mbq}$ to $X(\bbr)$ is a
positive-definite hermitian metric.

\item As a section of $(L\otimes \overline{\sigma^*L})^\vee $,
$\mbq$ is parallel with respect to the connection induced by
$\nabla.$
\end{enumerate}
One sees that the tensor product endows $\rlbc{X}$ with a group structure which makes $\rlbc{X}$ the kernel of $\Psi$:
\smallskip

\noindent{\bf Theorem \ref{thm:h22}.}\ {\it
If $X$ is a smooth real projective variety then  one has a natural
short exact sequence
$$
0\to \rlbc{X} \to \dcr{2}{2}{X} \xrightarrow{\Psi} \sigtor(X) \to 0.
$$
}
\smallskip

This paper is organized as follows. Section \ref{sec:ord} contains background information and the key technical ingredients for the paper, including Proposition \ref{prop:tau} which constructs the map  of complexes $\tau_p \colon \brAp \longrightarrow \cale^*$. In Section \ref{sec:del} we construct Deligne complexes $\dAp$ for any subring $A\subset \bbr$ and define the corresponding Deligne cohomology for proper smooth real varieties. In this section we prove basic properties, introduce the product structure and provide basic examples. In Section \ref{sec:exp} we study the weight $p=1$ case, proving the quasi-isomorphism  $\de{\bbz}{1}\simeq \calo^\times [-1]$ together with some applications. In Section \ref{sec:h22} we study the group $ \dcr{2}{2}{X}$ and associated interpretations of
$\bcz{2}{2}{X(\bbc)}$. The proof of the main result, Theorem \ref{thm:h22}, is delegated to Appendix \ref{app:proof}. Section \ref{sec:ex} contains a remark about number fields, where we give a ring homomorphism from the Milnor $K$-theory of a number field to the ``diagonal'' subring of integral Deligne cohomology, and we observe that the classical regulator of a number field can be described in terms of the image of the change-of-coefficients
homomorphism between Deligne cohomology with integral and real coefficients, respectively; similar computations can be made for arbitrary Artin motives. In Appendix \ref{sec:appEVv} we describe the relationship between Esnault-Viehweg's ``Borel version'' of Deligne cohomology for real varieties and the theory discussed in this paper.

In a forthcoming series of papers we first extend this theory to an \emph{integral Deligne-Beilinson cohomology theory} for arbitrary real varieties, and study its relation to a corresponding notion of \emph{Mixed Hodge Structures}. Then we provide a natural and explicit \emph{cycle map} from motivic cohomology to Deligne cohomology, directly using the approach in \cite{maz&wei}. This is an alternative description, even in the complex case, of the maps discussed in \cite{blo-acbc}, \cite{KLM-ajI}, \cite{ker&lew-ajII}. The corresponding real intermediate Jacobians and their relations to real algebraic cycles are also under study.
\medskip

\begin{ack}
 The first author would like to thank Texas A\&M University and the second author would like to thank the IST (Instituto  Superior T\'ecnico, Lisbon) for their respective warm hospitality during the elaboration of parts of this work; and the second author wants to thank Spencer Bloch for inspiring conversation and pointed questions during a visit to the University of Chicago.
\end{ack}

\section{Ordinary equivariant cohomology and sheaves}
\label{sec:ord}

We first introduce the various categories used throughout this
article. Let $\frS := Gal(\bbc/\bbr)$ denote the Galois group of
$\bbc$ over $\bbr.$

\noindent{\bf a)}\ The category of smooth manifolds with smooth
$\frS$-action and equivariant smooth morphisms is denoted $\gman$.
Let $\cov{X}$ be the set of coverings of $X\in \gman$ by open
$\frS$-invariant subsets. These coverings give $\gman$ a
site structure whose restriction to $X$ is denoted $X_{eq}.$ \\
\noindent{\bf b)}\  A \emph{real holomorphic manifold} $(M,\sigma)$ is a smooth
complex holomorphic manifold $M$ endowed with an
anti-holomorphic involution $\sigma \colon M \to M$, and morphisms
between two such objects are holomorphic equivariant maps. These
comprise the category $\rav$ of real holomorphic
manifolds, with its evident site structure. \\
\noindent{\bf c)}\  We denote by $\srv$ the category of smooth
real algebraic varieties.

\noindent{\bf d)}\  The categories of smooth holomorphic manifolds
and complex algebraic varieties are denoted $\cav$ and $\scv,$
respectively.
\smallskip

\begin{remark}
\label{rem:notation}
\begin{enumerate}
\item If $M$ is a complex manifold, denote  $M_{/\bbr}:= M \amalg
\overline{M}$, where $\overline{M}$ denotes $M$ with the opposite
complex structure. The map $\sigma \colon M_{/\bbr} \to M_{/\bbr}$
sending a point in one copy of $M$ to the same point in the other
copy of $M$ is an anti-holomorphic involution. The assignment
$M\mapsto M_{/\bbr}$ is a functor $\cav \to \rav$.
\item Given a real variety $X \in \srv$, let $X(\bbc)$ denote its set of complex-valued
points with the analytic topology. The natural  action of $\frS$
on $X(\bbc)$ which a morphism of sites from $\srv$ into $\rav$.
\item We will denote a \emph{complex} algebraic variety always as $X_\bbc$,
and we use $X_{\bbc/\bbr}$ to denote $X_\bbc$ seen as a real
variety. It follows that the set of complex valued points
$X_{\bbc/\bbr}(\bbc)$ (over $\bbr$) coincides with $X_{\bbc}(\bbc)
\amalg \overline{X_{\bbc}(\bbc)}$. The constructions above give a
commuting diagram of functors:
$$
\xymatrix{ \scv \ar[r] \ar[d] & \srv \ar[d] \\
\cav \ar[r] & \rav.
 }
$$
\end{enumerate}
\end{remark}

\subsection{Ordinary equivariant cohomology}

In \cite{bredon} Bredon defines an equivariant cohomology theory
$H^n_G(X;M)$ for $G$-spaces, where $G$ is a finite group and $\uM$
is a contravariant coefficient system. When $\uM$ is a Mackey
functor, P. May et al. \cite{may-bull} showed that this theory can be
uniquely extended to an $RO(G)$-graded theory $\{
H^\alpha_G(X;\uM), \alpha \in RO(G)\}$, called
\emph{$RO(G)$-graded ordinary equivariant cohomology theory},
where $RO(G)$ denotes the orthogonal representation ring of $G$.
When $G=\frS$, one has $RO(\frS) = \bbz \cdot\bone \oplus
\bbz\cdot \xi,$ where $\bone$ is the trivial representation and
$\xi$ is the sign representation. In this paper we use the
\emph{motivic notation}:
\begin{equation}
\label{eq:bredon}
\bcm n p X \uM :=  \ H^{(n-p)\cdot \bone + p\cdot
\xi}_\frS(X;\uM),
\end{equation}
and call $\bcm n p X \uM $ \emph{bigraded Bredon cohomology}.

In the homotopy theoretic approach, one proves the existence of
equivariant Eilenberg-MacLane spaces $K(\uM, (n,p) )$ that
classify Bredon cohomology (for $n\geq p\geq 0$). In other words,
$\bcm{n}{p}{X}{\uM} = [X_+,K(\uM,(n,p))]_G$, where the latter
denotes the set of based equivariant homotopy classes of maps.
When $n<p$ one uses the suspension axiom to define the
corresponding cohomology groups; see \cite{may-bull}.

A quick way to construct $K(\uZ,(n,p))$ is the following. Let
$S^{n,p}$ denote the one-point compactification $\{ (n-p)\cdot
\bone \oplus p\cdot \xi \} \cup \{ \infty \}$ of the indicated
representation, and define $\bbz_0( S^{n,p}):= \bbz(S^{n,p})/\bbz
(\{ \infty \}) $, where $\bbz(S^{n,p})$ denotes the free abelian
group on $S^{n,p},$ suitably topologized. Then $\bbz_0(S^{n,p})$
is an Eilenberg-MacLane space $K(\uZ,(n,p));$ cf. \cite{dS-DT}.

\begin{example}
\label{ex:br0-pt}
In order to describe the bigraded cohomology ring of a point $\calb:=\oplus_{p,n}\ \bcm{n}{p}{pt}{\uZ},$ first consider indeterminates $\ve, \ve^{-1}, \tau, \tau^{-1}$ satisfying $\deg{\ve}=(1,1),\
\deg{\ve^{-1}}=(-1,-1),\ \deg{\tau}=(0,2)$ and
$\deg{\tau^{-1}}=(0,-2).$ Henceforth, $\ve$ and $\ve^{-1}$ will
always satisfy $2\ve=0=2\ve^{-1}.$

As an abelian group, $\calb$ can be written as a direct sum
\begin{equation}
\label{eq:M}
\calb \ \equdef \ \bbz[\ve,\tau]\cdot 1 \ \oplus \ \bbz[\tau^{-1}]
\cdot \alpha  \oplus \ \bbf_2[\ve^{-1},\tau^{-1}] \cdot \theta
\end{equation}
where each summand is a free bigraded module over the
corresponding ring and $\bbf_2$ is the field with two elements (hence $2\theta = 0$). The respective bidegrees of the generators  $1$, $\alpha$ and $\theta$ are\ $(0,0)$, $(0,-2)$ and $(0,-3).$

The product structure on $\calb$ is completely determined by the
following relations
\begin{equation}
\label{eq:rels}
\alpha\cdot \tau = 2, \quad \quad \alpha\cdot \theta =\alpha\cdot
\ve = \theta \cdot \tau = \theta \cdot \ve \ =\ 0 .
\end{equation}
Note that $\calb$ is not finitely generated as a ring, and that
$\calb$ has no homogeneous elements in degrees $(p,q)$ when
$p\cdot q < 0$.
\end{example}

We now present an alternative sheaf-theoretic construction of
Bredon cohomology which is more suitable for our purposes. The
details of such construction will appear in \cite{yang-pre}.

Given $U\in \gman$, let $\widehat U$ denote the full subcategory
of $\gman\! \downarrow U$ consisting of equivariant finite
covering maps $\pi_S\colon \ S\to U$. In particular,
$\widehat{\{pt\}}$ is the category $\gfin \subseteq \gman$ of
finite $\frS$-sets..

A topological $G$-Module $M$ represents an abelian \emph{Mackey
presheaf} on $\gman$, in other words, the contravariant functor
$\uM \colon \gman^{op}\to Ab$ sending $U\mapsto \uM(U) :=
\text{Hom}_{\gtop}(U,M)$ is also covariant for maps in $\widehat
U$, for all $U\in \gman$, and satisfies the following property.
Given a pull-back square
\[
\begin{CD}
Z @>\gamma>> X\\
@V{\varphi}VV @VV{f}V\\
Y @>>g> U
\end{CD}
\]
with $f\in \widehat U$, the diagram
\[
\begin{CD}
\uM(Z) @<{\gamma^*}<< \uM(X)\\
@V{\varphi_*}VV @VV{f_*}V\\
\uM(Y) @<<{g^*}< \uM(U)
\end{CD}
\]
commutes. The case where $M$ is a subring a $\bbr$ with trivial
$\frS$-module structure will play a special role in this work.

\begin{definition}
Let ${\calf}$ be an abelian presheaf on $\gman$, and let $M$ be a
topological $\frS$-module. Given $U\in\gman$, let
${\calf}\otimes_{\widehat U}\uM$ denote the coend:
$$
\{ \bigoplus_{\pi_S\colon S \to U} {\calf}(S)\otimes \uM(S)\ \}\
/\ K_{{\calf},\uM}(U) \quad \text{in} \quad \ab ,
$$
where
$\pi_S \in \widehat{U}$ and ${\calk}_{{\calf},\uM}(U)$ is the
subgroup generated by elements of the form:
\[
(\phi^*_{\calf}\alpha') \otimes m - \alpha' \otimes \phi_{\uM^*}m,
\]
when\
$\begin{matrix} S~~\overset{\sst
\phi}{\longrightarrow}~~S'\\ {}_{\pi_S}\searrow~~\swarrow
{}_{\pi_{S'}}\\ U\end{matrix}$ is a morphism in $\widehat U$,
$\alpha'\in {\calf}(S')$\  and $m\in \uM(S)$. It is easy to see
that the assignment $U\longmapsto {\calf} \otimes_{\widehat U}
\uM$ is a contravariant functor from $\gman$ to $Ab$. Denote by
${\calf}\coend \uM$ the resulting abelian presheaf on $\gman$,
i.e.\ \ ${\calf}\coend\uM(U) := {\calf}\otimes_{\widehat U} \uM$.
\end{definition}

Let $ \vD^n $ denote the standard topological $n$-simplex with the
trivial $\frS$-action. Using the co-simplicial structure on $\{
\vD^*\mid n \geq 0 \}$ one creates a simplicial abelian presheaf
$\calc_{\bullet}(\calf)$ associated to any presheaf $\calf$, whose
$n$-th term is
$$
\calc_n(\calf)\ \colon\ U\longmapsto\ \calf(\vD^n \times U).
$$
Denote the associated complex of sheaves by $\left(
\calc_*(\calf), d_* \right)$ and use the convention in
\cite[XVII 1.1.5]{SGA4}
to define a cochain complex $\left( \calc^*(\calf), d^*
\right)$ where $\calc^n(\calf) := \calc_{-n}(\calf)$ and $d^n
\colon \calc^n(\calf) \to \calc^{n+1}(\calf)$ is defined by $d^n =
(-1)^nd_{-n}.$

A $\frS$-manifold $X$ defines an abelian presheaf on $\gman$
$$
\calz X \colon \ U \longmapsto \bbz \hom{\gman}{U}{X},
$$
sending $U$ to the free abelian group on the set of smooth
equivariant maps from $U$ to $X.$ Yoneda Lemma identifies $\hom{AbPreSh}{\calz X}{\calf} = \calf(X)$ for any
$\calf.$ In particular, if $\calf$ is any abelian presheaf and
$X\in \gman$, the presheaf $\underline{Hom}(\calz X, \calf)$ sends
$U$ to $\calf(X\times U).$

\begin{proposition}
\label{rem:simp-ps}
\begin{enumerate}[i.]
\item The assignment $\calf \mapsto \calc^*(\calf;\uM)$ is
covariant on $\calf$.
\item Let $I=[0,1]$ denote the unit interval with the trivial
$\frS$-action. For any abelian presheaf $\calf$ let $i_0^*, i_1^*
\colon \underline{Hom}(\calz I, \calc^*(\calf)) \to
\calc^*(\calf),$ be the map of complexes induced by
\emph{evaluation at the end-points}. Then there is a homotopy
$\mbh_\calf$ between $i_0^*$ and $i_1^*$ which is natural on
$\calf.$ In particular, the complexes $\calc^*(\calf)$ have
\emph{homotopy-invariant} cohomology presheaves.
\end{enumerate}
\end{proposition}

In what follows denote
$$\csi{p-1} := \cs\times \cdots \times 1 \times \cdots \times \cs \subset
\cs^p,$$ where $1$ appears in the $i$-th coordinate.

\begin{definition}
\label{def:bredon}
Given  a $\frS$-manifold $X$, let $$ {J}_{X,p}\ \colon
\bigoplus_{i=1}^p\ \calc^*(\calz( \csi{p-1}\times X )) \
\longrightarrow \ \calc^*( \calz( \cs^p \times X) ) $$ be the map
induced by the inclusions and denote
\begin{equation}
\label{eq:redX}
 \calc^*(\calz_0(S^{p,p}\wedge X_+)) := \cone(J_{X,p}).
\end{equation}
Write $\redcxp$ when $X=\emptyset$ is the empty manifold. Let
$A\subset \bbr$ be a subring endowed with the discrete topology.
The \emph{$p$-th Bredon complex with coefficients in $A$} is the
complex of presheaves
\begin{equation}
\label{eq:pBrCx}
\brAp := \redcxp \coend \uA\ [-p] \ ,
\end{equation}
where the coend is taken levelwise.
\end{definition}
The following result is proven in \cite{yang-pre}

\begin{theorem}
Let $X$ be a $\frS$-manifold and let $A\subset \bbr$ be a subring,
endowed with the discrete topology. Then for all $p\geq 0$ and
$n\in \bbz$ there is a natural isomorphism $\check{\bbh}^n(X_{eq};
{\brAp} ) \cong \bca{n}{p}{X}$ between the $\check{\text{C}}$ech
hypercohomology of $X_{\text{eq}}$ with values in $\brAp$ and the
equivariant cohomology group $\bca n p X $.
\end{theorem}

For all $p,n \in \bbz$ there is a \emph{forgetful functor}
\begin{equation}
\label{eq:forgetful}
\varphi \colon \bca n p X \  \longrightarrow \
\scm{n}{X}{A(p)}^\frS,
\end{equation}
where $A(p)$ is the $\frS$-submodule $A(p):= (2\pi i)^pA \subset
\bbc$ and the invariants $\scm{n}{X}{A(p)}^\frS$ of the
singular cohomology of $X$ with coefficients in $A(p)$ are taken under the
simultaneous action of $\frS$ on $X$ and $A(p).$

In the following section we present a realization of the
composition
\begin{equation}
\label{eq:composition}
\bca n p X  \to \scm{n}{X}{A(p)}^\frS \to \scm{n}{X}{\bbc}^\frS,
\end{equation}
where the latter is the change of coefficients map.

\subsection{Integration and change of coefficient functors}

Let $\cala^p$ denote the sheaf of smooth, complex valued
differential $p$-forms on $\frS$-manifolds. Given $X\in \gman$,
denote
$$
\cale^p(X)= \{ \theta \in \cala^p(X) \mid
\overline{\sigma^*(\theta)} = \theta \}.
$$
In other words, $\cale^p$ is the subsheaf of $\cala^p$ consisting
of those $p$-forms invariant under the simultaneous action of
$\frS$ on $X$ and $\bbc.$

Let $\pi \colon E \to B$ be a locally trivial bundle, where $B$ be a smooth manifold, $E$ is an oriented manifold-with-corners and the fiber is an orientable $n$-dimensional
manifold-with-boundary $F$ (and with corners). Let
\begin{equation}
\label{eq:integration}
\pi_! \colon \cala^{p+n}(E) \to \cala^p(B).
\end{equation}
denote the \emph{integration along the fiber} homomorphism.  If
$\pi$ is a map in $\gman$ then $\pi_!$ preserves invariants, thus
sending $\cale^{p+n}(E)$ to $\cale^p(B).$ The following properties
are well-known.
\begin{properties}
\label{proper:int}
Let $\omega$ be a $(n+p)$-form and let $\pi' \colon \partial E \to
X$ denote the restriction of $\pi$ to the boundary of $E$.
\begin{enumerate}[i.]
\item(Projection formula)\label{eq:proj} $\pi_!(\pi^* \theta \wedge \omega) = \theta \wedge \pi_!(\omega)$
\item(Boundary formula) \label{eq:bdry} $d \pi_!(\omega) = \pi_!(d\omega) + (-1)^p \pi'_!(\omega_{|_{\partial E}}) $
\item(Pull-back formula) Given a pull-back square
$$
\xymatrix{ E' \ar[r]^{f'} \ar[d]^{\pi'} & E \ar[d]^{\pi} \\
X' \ar[r]_{f} & X},$$ then $f^* \circ \pi_! \ = \ \pi'_{!} \circ
f'^*.$
\item(Functoriality) If $X \xrightarrow{f} Y \xrightarrow{g} Z$ are smooth
fibrations with compact fibers, then $(g\circ f)_! = g_! \circ
f_!$.
\item(Product formula)\label{eq:prod} Let
$$
\xymatrix{ E'' \ar[r]^{\rho} \ar[d]_{\rho'} \ar@{-->}[dr]^{\pi''} & E \ar[d]^{\pi} \\
E' \ar[r]_{\pi'} & X},
$$
be a pull-back square where both $\pi$ and $\pi'$ are fibrations
with fiber dimensions $n$ and $n'$, respectively. Given $\omega\in
\cala^{p+n}(E)$ and $\omega'\in \cala^{q+n'}(E')$ one has
$\pi''_!( \rho_1^*\omega \wedge \rho^*_2\omega') = (-1)^{nq}
\pi_!(\omega) \wedge \pi'_!(\omega').$
\end{enumerate}
\end{properties}

Integration along the fiber can be used to construct maps of
complexes
\begin{equation}
\label{eq:tau0}
\tau_p \colon \brAp \to \cale^*,
\end{equation}
as follows. First consider  $U\in \gman$ and $0\le j\le p$. An
element in $\brAp^j(U)$ is represented by sums of pairs of the
form $\alpha \otimes m $, where $\alpha = (a,f)$ with $a, f$ and
$m$ equivariant maps satisfying
\begin{enumerate}[1.]
\item $a \colon \Delta^{p-j-1} \times S \to \csi{p-1} \subset \cs^p
$ is smooth and $\pi \colon S\to U $ is a map in $\widehat U$;
\item $f\colon\ \Delta^{p-j}\times S \to (\cs )^p$ is a smooth map;
\item $m\colon \ S\to A \in \uA(S)$ is a locally constant.
\end{enumerate}

In the diagram
\begin{equation*}
\xymatrix{   & & & & \cs^p \\  U  & \Delta^{p-j}\times U
\ar[l]_-{\ p_1\ } & \Delta^{p-j}\times S \ar[l]_-{\ 1 \times \pi\
} \ar@/^{1.2pc}/@{->}[ll]^{\hat{\pi} } \ar[rru]^f  \ar[dr]_{p_2} & & \\
& & & S \ar[r]_m & A  },
\end{equation*}
$\hat\pi$ is a locally trivial fibration with fiber dimension
$p-j$, and we consider the locally constant map $m\circ p_2 =
p_2^*m $ as an element in $\cale^0(\Delta^{p-j}\times S)$. Denote
\begin{equation}
\label{eq:omegap}
\omega_p \ := \ \frac{dt_1}{t_1} \wedge\cdots\wedge
\frac{dt_p}{t_p} \in {\cale}^p(\{\cs\}^p)
\end{equation}
and define
\begin{equation}
\label{eq:tau1}
\tau^j(\alpha\otimes m) = \hat\pi_{!}\{p_2^*m \cdot f^*\omega_p\}
\ \in \cale^j(U)
\end{equation}
where $\hat\pi_!\colon \ {\cale}^p(\Delta^{p-j}\times S)
\longrightarrow {\cale}^j(U)$ is the integration along the fiber
homomorphism. This can be extended to a homomorphism
\[
\tau^j \colon \ \bigoplus_{S\in\widehat U}
 \left\{ \calc^{p-j}(\calz_0(S^{p,p})) (S)\otimes
 \uA(S)\right\}
 \longrightarrow \cale^j(U).
\]
\begin{proposition}
\label{prop:tau}
For each $0\leq j \leq p$, the map $\tau^j$ above factors through
$\  \calc^{p-j}(\calz_0(S^{p,p}))\otimes_{\widehat{U}} \uA$.
Furthermore, these maps induce a morphism of complexes of
presheaves
\begin{equation}
\label{eq:tau}
\tau_p \colon \brAp \longrightarrow \cale^*.
\end{equation}
\end{proposition}
\begin{proof}
Given a morphism
$$
\xymatrix{ S \ar[rr]^{\phi} \ar[dr]_{\pi} & & S' \ar[dl]^{\pi'}
\\
 & U &
}
$$
in $\widehat{U}$, consider the associated diagram
\begin{equation}
\label{eq:assoc_diag}
\xymatrix{
  \Delta^{p-j} \times S \ar[rrrrrr]^{1\times \phi} \ar@/_{1.2pc}/@{->}[ddrrr]^{\hat{\pi}\ \ \ \ }
 \ar[drr]^-{p} & & & & & & \Delta^{p-j} \times
  S'  \ar[dll]_-{p'} \ar@/^{1.2pc}/@{->}[ddlll]^{\hat{\pi}' }  \ar@{-->}[dd]^{f}  \\
  & & S \ar@{-->}[dll]^-{m} \ar[rd]^{\pi} \ar[rr]^-{\phi} &   &  S' \ar[dl]_{\pi'} & & \\
  A & &  & U  & & & \{\cs \}^p.
}
\end{equation}
Pick $\alpha=(a,f) \in \cone(J_p)^{p-j}(S')$ and $m\in \uA(S)$;
cf. Definition \ref{def:bredon}. By definition,
\[
\tau^j(\phi^*\alpha\otimes m ) = \hat \pi_!(p^* m \cdot
(1\times \phi)^* f^*\omega_p).
\]
It follows from Properties i)--iv) of integration along the fibers
that for all $\theta \in \cale^*(\Delta^{p-j}\times S')$ one has
$\hat\pi_!\{p^*m \cdot(1\times \phi)^*\theta\} =
\hat\pi'_!(p'^*(\phi_* m) \cdot \theta)$, since the top square in
\eqref{eq:assoc_diag} is a pull-back diagram.
%
In particular, for $\theta = f^*\omega_p$ one obtains
$$
\tau^j(\alpha\otimes \phi_*m) =  \hat \pi'_! (p'^* (\phi_*m) \cdot
f^*\omega_p) = \tau^j(\phi^*\alpha \otimes m).
$$
This proves the first assertion in the proposition.

To prove the second assertion, let $\partial_i \colon
\Delta^{p-j-1} \hookrightarrow \Delta^{p-j}$ be the inclusion of
the $i$-th face, and pick $\alpha\otimes m \in
\cone(J_p)^{p-j}(S)\otimes \uA(S)$, as before, representing an
element in $\brAp^j(U)$. Then, using the sign convention relating
the differential $D$ of $\brAp^*$ with the differential $d^B$ of
$\calc_*(\calz (\cs^p))$ one gets
\begin{multline}
\label{eq:eq1}
\tau^{j+1}( D(\alpha \otimes m) )  = \tau^{j+1}\left( (-d^*_p a,
a+d^*_p f)\otimes m \right) \\
 =  \hat \pi_! \left( p^*m \cdot a^*
\omega_p \right)\ + \ (-1)^j \sum_{i=1}^n (-1)^i \hat \pi_! (
\tilde p^*m \cdot \partial_i^* f^* \omega_p) ,
\end{multline}
where $\tilde p \colon \Delta^{p-j-1} \times S \to S$ is the
projection. Since $a \colon \Delta^{p-j-1}\times S \to \cs^p$
factors through some $\csi{p-1} \subset \cs^p$, one has $a^*
\omega_p = 0.$ On the other hand,
\begin{multline}
\label{eq:eq2}
\sum_{i=1}^n (-1)^i \hat \pi_! ( \tilde p^*m \cdot \partial_i^*
f^* \omega_p)
 = \sum_{i=1}^n (-1)^i \hat \pi_! ( (\partial_i \times 1)^* \{ p^*m
 \cdot f^*\omega_p \} ) \\
 = \hat \pi_! ( \{ p^*m \cdot f^*\omega_p \}_{|_{\partial(
 \Delta^{p-j}\times S) }} )
 = (-1)^{j} d \hat \pi_! (p^*m \cdot f^*\omega_p);
\end{multline}
cf. the \emph{boundary formula} \eqref{proper:int}\ref{eq:bdry}.
It follows from \eqref{eq:eq1} and \eqref{eq:eq2} that $\tau^{j+1}
\circ D = d \circ \tau^j, $ for $1\leq j \leq p.$

The two remaining cases are:
\begin{equation}
\label{eq:extreme1}
\xymatrix{
0 \ar[r] & \cale^{p+1}(U) \\
\brAp^p(U) \ar[u] \ar[r]_{\tau^p} & \cale^p(U) \ar[u]^{d}}
\end{equation}
and
\begin{equation}
\label{eq:extreme2}
\xymatrix{\brAp^0(U) \ar[r]_{\tau^0} & \cale^0(U) \\
\brAp^{-1}(U)\ar[u]_{d^{-1}} \ar[r] & 0 \ar[u]}.
\end{equation}

Given $f\colon \ S\to \cs^p$, $m\colon \ S\to A$ with $S\in
\widehat U$, then the boundary formula for integration along the
fibers gives $d\ \pi_!(m\cdot f^*\omega_p) = 0$, thus showing that
\eqref{eq:extreme1} commutes.

To show that \eqref{eq:extreme2} commutes, pick $f\colon \ S\times
\Delta^{p+1} \longrightarrow \cs^p$, $m\colon \ S\to A$. Then:
\begin{align*}
\tau_0(D[f\otimes m]) &= \tau_0 \left\{\sum^{p+1}_{k=0} (-1)^k f\circ
(1\times\partial_k) \otimes m\right\}\\
&= \sum^{p+1}_{k=0} (-1)^k \hat\pi_! (m\cdot(1\times\partial_k)^*
f^*\omega_p) \\ & =
\sum^{p+1}_{k=0} (-1)^k \hat\pi_!( (1\times\partial_k)^* \{m\cdot
f^*\omega_p\})\\
&= d\hat\pi_!(m\cdot f^*\omega_p) \pm \hat\pi_!(d\{m\cdot
f^*\omega_p)) = 0.
\end{align*}

The functoriality of the maps $\tau^j$ with respect to $U$ should
be evident.
\end{proof}

\begin{remark}
\label{rem:eq-vs-nonequ}
Recall that if $U$ is a $\frS$-manifold, one obtains a natural
$\frS$-isomorphism $U^{\rm triv} \times \frS \xrightarrow{\cong} U
\times \frS$ by sending $(x,g)$ to $(gx, g)$. On the other hand,
$\cale^p(U^{\rm triv} \times \frS)$ is the subgroup of $\cala^p(U^{\rm
  triv} \times \frS) \cong \cala^p( U^{\rm triv} ) \times
\cala^p(U^{\rm triv})$ invariant under the involution that
sends $(\omega_1,\omega_\sigma)$ to \
$(\overline{\omega_\sigma},\overline{\omega_1}).$ This observation
shows that we have a functor
\begin{align*}
F\colon \cale^p(U\times \frS) & \longmapsto \cala^p(X)\\
\omega & \longmapsto \omega_{|_{ U\times 1}}
\end{align*}
satisfying the following properties:
\begin{enumerate}
\item $F$ commutes with differentials;
\item If $\imath \colon U\times \frS \to U\times \frS$ is the $\frS$-homeomorphism sending $(x,\alpha)$ to $(x, \sigma\alpha)$, then $F (\imath^* \omega) = \overline{\sigma^* F(\omega)}$;
\item $F\circ p^* = Id$, where $p\colon U\times \frS \to U$ is the projection.
\end{enumerate}
\end{remark}

To describe the product structure on $A(*)_{\mathcal{B}r} $, let $\Gamma_{n,m} := \{ \sigma \colon \Delta^{n+m} \to
\Delta^n\times \Delta^m \ | \ n\geq 0,\ m \geq 0\}$ be the
triangulation of $\Delta^n\times \Delta^m$ inducing the
Alexander-Whitney diagonal approximation; cf. \cite[p. 68]{eil&ste-FAT}. Given $p,q\geq 0$, the
(external) pairing of presheaves $ \calz(\cs^p) \otimes
\calz(\cs^q) \to \calz(\cs^{p+q}) $ yields a pairing of complexes
\begin{equation}
\label{eq:pairing1}
\calc^*(\calz(\cs^p)) \otimes \calc^*(\calz(\cs^q)) \to
\calc^*(\calz(\cs^{p+q}))
\end{equation}
in the usual manner. Denoting $\calc^*(\calz(\widehat{\cs^r})) :=
\oplus_{i=1}^r\ \calc^*(\calz(\csi{r-1})),$ one sees that
this pairing sends both $\calc^*(\calz(\widehat{\cs^p}))
\otimes \calc^*(\calz(\cs^q)) $ and $\calc^*(\calz(\cs^p)) \otimes
\calc^*(\calz(\widehat{\cs^q}))$ to
$\calc^*(\calz(\widehat{\cs^{p+q}}))$, and hence it induces a
pairing of complexes
\begin{equation}
\label{eq:pairing2}
\calc^*(\calz_0(S^{p,p})) \otimes \calc^*(\calz_0(S^{q, q})) \to
\calc^*(\calz_0(S^{p+q,p+q})).
\end{equation}
Finally, the multiplication $\uA\otimes \uA \to \uA$ together with
the appropriate sign conventions yields a pairing of complexes $
\mu \colon \brAp \otimes \brAq \to \br{A}{p+q}.$ See the proof of
Theorem \ref{thm:mult} below.

\begin{theorem}
\label{thm:mult}
The maps of complexes $\tau$ are compatible with multiplication.
In other words, for every $p, q \geq 0$ and $i,j \in \bbz$ one has
a commutative diagram of presheaves:
$$
\xymatrix{ \brAp^i  \ \otimes \ \brAq^j \ar[r]^-{\mu}
\ar[d]_{\tau_p
\otimes \tau_q}  & \br{A}{p+q}^{i+j} \ar[d]^{\tau_{p+q}} \\
\cale^i \ \otimes \ \cale^j \ar[r]_{\wedge} & \cale^{i+j} }
$$
\end{theorem}
\begin{proof}
Given $U\in \gman,$ elements in $\brAp^i(U)$ and $\brAq^j(U)$ are
represented respectively by $\alpha\otimes m = (a,f)\otimes m $
and $\beta\otimes n = (b,g)\otimes n $, where
\begin{enumerate}
\item $\pi_S \colon S \to U$ and $\pi_T \colon
T\to U$ are in $\widehat U$ and  $\pi \colon S\times _U T \to U$
denotes their fibered product;
\item $a \colon \Delta^{p-i-1}\times S \to {\bbc_{r}^\times}^{p-1}$,\  $f \colon \Delta^{p-i} \times S \to
\cs^p$ smooth equivariant, and $m \colon S \to A$ locally
constant;
\item $b \colon \Delta^{q-j-1}\times T \to {\bbc_{s}^\times}^{q-1}$,  $g \colon \Delta^{q-j} \times T \to
\cs^q$ smooth and equivariant, and $n \colon T \to A$ locally
constant.
\end{enumerate}
The relevant maps in the constructions that follow are summarized
in the diagram

{\small
$$
\xymatrix{
\Delta^{p+q-i-j} \times S\times_U T  \ar[dd]_{\sigma\times 1}
\ar `r[rrrr] ^{\hat\sigma} [rrrrdd]  & & & & \\
& \Delta^{p-i}\times S \ar[r]^-{p_S} \ar[rrrd]^-{\hat\pi_S} & S \ar[r]^m & A  & \\
\Delta^{p-i}\times \Delta^{q-j}\times S\times_U T \ar[d]_{i}
\ar[ru]^{\rho_S} \ar[rd]_{\rho_T} \ar[ru]
\ar@{-->}[rr]^{p}& &  S\times_U T \ar[rr]^\pi &  & U\\
(\Delta^{p-i}\times S) \times (\Delta^{q-j}\times T)
\ar[d]_{f\times g} &
\Delta^{q-j}\times T \ar[rrru]_-{\hat\pi_T} \ar[r]_-{p_T} & T\ar[r]_n & A  & \\
\cs^p \times \cs^ q \equiv \cs^{p+q} }
$$
}

where $\sigma \colon \Delta^{p+q-i-j} \to \Delta^{p-i}\times
\Delta^{q-j} $ is in $\Gamma_{p-i,q-j}.$

By definition, $\mu\left( \alpha\otimes m,  \beta\otimes n\right)
= (a\cup_U g + f\cup_U b,\ f\cup_U g) \otimes m\star n,$ where $
f\cup_U g := (-1)^{j(i+p)} \sum_{\sigma\in\Gamma_{p-i,q-j}}\
(-1)^{|\sigma|}\ (f\times g) \circ i \circ (\sigma\times 1),$ and
$m\star n \colon S\times_U T \to A$ is defined as $m\star n (s,t)
= m(s)n(t) \in A.$  The elements $a\cup_U g$ and $f\cup_U b$ are
defined similarly.

Hence
\begin{multline}
\label{eq:compat}
(-1)^{j(i+p)}\ \tau_{p+q}( \mu\left( \alpha\otimes m,  \beta\otimes
n\right))= \\
= \sum_{\sigma} (-1)^{|\sigma|}  \hat\sigma_! \left(
\{ (\sigma\times 1) \circ p\} ^*(m\star n) \cdot \{ (f\times g)
\circ i \circ
(\sigma\times 1) \}^*\omega_{p+q}      \right) \\
=\sum_{\sigma} (-1)^{|\sigma|}  \hat\sigma_! \left(
(\sigma\times 1)^*\left\{ p^* (m\star n) \cdot
 i^* \circ (f\times g)^* \omega_{p+q}  \right\} \right) \\
=\sum_{\sigma} (-1)^{|\sigma|}  \hat\sigma_! \left( (\sigma \times
1)^* \{ p^*(m\star n) \cdot   i^* (f^*\omega_p \times g^*
\omega_{q} )  \} \right) .
\end{multline}
The collection $\Gamma_{p-i,q-j}$ gives a triangulation of
$\Delta^{p-i}\times \Delta^{q-j}$ and hence
\begin{multline}
\label{eq:mid}
 \sum_{\sigma} (-1)^{|\sigma|}  \hat\sigma_! \left( (\sigma
\times 1)^* \{ p^*(m\star n) \cdot   i^* (f^*\omega_p \times g^*
\omega_{q} )  \} \right)  \\
= (\pi\circ p)_!\left(  p^*(m\star n) \cdot   i^* (f^*\omega_p
\times g^* \omega_{q} )     \right) =
(\pi\circ p)_!\left(  \rho_S^*a_S \wedge\ \rho_T^* a_T \right),
\end{multline}
where $a_S := p_S^*m \cdot f^*\omega_p$ and $a_T := p_T^*m \cdot
g^*\omega_q$. The \emph{product formula} (see Properties
\ref{proper:int}(\ref{eq:prod})) gives $(\pi\circ p)_!\left(
\rho_S^*a_S \wedge\ \rho_T^* a_T \right) = (-1)^{(p-i)j}\
\hat\pi_{S_!}(a_S) \wedge\ \hat\pi_{T_!}(a_T) := (-1)^{(p-i)j}
\tau_p(\alpha\otimes m) \wedge \tau_q(\beta\otimes n) $. The
theorem now follows from this last identity together with
\eqref{eq:compat} and \eqref{eq:mid}.
\end{proof}

\begin{remark}
The forgetful functor $\phi \colon \bca * * X   \to
\scm{*}{X}{\bbc}^\frS $ described in \eqref{eq:composition} is the
ring homomorphism induced by the morphisms of complexes $\tau_*
\colon A(*)_{\text{Br}} \to \cale^*.$
\end{remark}

\section{Deligne cohomology for real varieties}
\label{sec:del}

In this section we restrict our attention to $\rav$, the category
of real holomorphic manifolds, seen as a site with the topology
induced by $\gman$. In particular, heretofore the complexes
$\brAp, \cale^*$ and $\cala^*$ will be restricted to $\rav$.

The Hodge decomposition $\cala^n = \oplus_{i+j=n}\ \cala^{i,j}$ is
invariant with respect to the $\frS$-action on $\cala^n$ and this
gives a Hodge filtration $\{ F^p\cale^* \}$ on the complex
$\cale^*$ of invariant smooth forms.

\begin{definition}
\label{def:Deligne}
Let $A\subset \bbr$ be a subring endowed with the discrete
topology.
\begin{enumerate} [i.]
\item Given $p\geq 0,$ define the $p$-th \emph{equivariant Deligne
complex} $\dAp$ on $\rav$ as:
\begin{equation}
\label{eq:del_comp}
\dAp := \cone\left( \brAp \oplus F^p\cale^* \xrightarrow{\iota_p}
\cale^*\right)[-1],
\end{equation}
where $\iota_p(\alpha, \omega) = \tau_p(\alpha) - \omega$; cf.
\eqref{eq:tau}.
\item Given a proper manifold $X \in \rav$ and $p\geq 0,$ define
the \emph{Deligne cohomology} of $X$ as the $\check{\text{C}}$ech
hypercohomology groups
\begin{equation}
\label{eq:del_coh}
\dcrA{i}{p}{X} := \check\bbh^i\left( X_\text{eq}; \dAp\right),
\end{equation}
where $X_\text{eq}$ denotes the equivariant site of $X$.
\item If $p<0,$ define $\dcrA{i}{p}{X}:= \bca{i}{p}{X(\bbc)}.$ In other words, Deligne
cohomology with negative weights is defined so as to coincide with
Bredon cohomology.
\item The following diagram introduces notation for the various
natural maps arising from the cone \eqref{eq:del_comp}.

{\tiny
\begin{equation}
\label{eq:forget}
\xymatrix{
& \bca{n}{p}{X(\bbc)} \ar[r]^{\varphi} &
H^n_\text{sing}(X(\bbc);A(p))^\frS \ar[d]^{\jmath} \\
\dcrA{n}{p}{X} \ar[ur]^{\varrho} \ar[r]_-{\nu} &
\bca{n}{p}{X(\bbc)} \oplus F^pH^n_\text{sing}(X(\bbc);\bbc)^\frS
\ar[r] \ar[u]^{pr_1} \ar[d]_{pr_2} & H^n_\text{sing}(X(\bbc);\bbz(p))^\frS \\
& F^pH^n_\text{sing}(X(\bbc);\bbc)^\frS \ \ \ \ar@{^{(}->}[ru] }
\end{equation}
}

\end{enumerate}
\end{definition}

\begin{proposition} Let $Y_\bbc$ be a proper complex holomorphic
  manifold and let $X$ be a proper smooth real algebraic variety.
\label{prop:change}
\begin{enumerate}[i.]
\item One has  natural isomorphisms
$$
\dcrA{i}{p}{Y_{\bbc/\bbr}} \cong \dccA{i}{p}{Y_\bbc},
$$
where the latter denotes the usual Deligne cohomology of the complex
manifold $Y_\bbc$. (See Remark \ref{rem:notation}.)
\item If $X_\bbc$ is the complex variety obtained from $X$ by base
  extension, the corresponding map of real varieties $X_{\bbc/\bbr}
  \to X$ induces natural homomorphisms
$$
\dcrA{i}{p}{X} \to  \dccA{i}{p}{X_\bbc}^\frS,
$$
where the latter denotes the invariants of the Deligne cohomology of
the complex variety $X_\bbc$.

\item  One has a long exact sequence
{\small \begin{multline}
\cdots \to H_\text{sing}^{j-1}(X(\bbc),\bbc)^\frS \to \dcrA{j}{p}{X}\xrightarrow{\nu} \\
\to \bca{j}{p}{X(\bbc)} \oplus \left\{ F^pH_\text{sing}^j(X(\bbc);\bbc)
\right\}^\frS \xrightarrow{\ \ \jmath\circ \varphi - \imath \ \ }
H_\text{sing}^j(X(\bbc);\bbc)^\frS \to \cdots
\end{multline}
}
\end{enumerate}
\end{proposition}

\begin{remark}
The map $\dcrA{i}{p}{X} \to  \dccA{i}{p}{X_\bbc}^\frS,$ is not
  an isomorphism in general. However, it is always an isomorphism when
  $1/2 \in A$.
\end{remark}

\begin{example}
\label{ex:dc-point}
Denote
\begin{align*}
\cald^{i,p} & := \dcr{i}{p}{\spec{\bbr}}\\
\cald_\bbr^{i,p} & := \dcrr{i}{p}{\spec{\bbr}} \quad \text{and} \\
\cald_\bbc^{i,p} & := \dcc{i}{p}{\spec{\bbr}},
\end{align*}
and let
$\cald^{i,p} \to \calb^{i,p} := \bcz{i}{p}{pt}$ denote the natural map between the Deligne and
Bredon cohomology groups of a point, respectively. See Example~\ref{ex:br0-pt}.
The following statements follow from the exact sequence in the Proposition above
and from the definitions. Fix $p\geq 0.$
\begin{enumerate}[i.]
\item The vanishing of singular cohomology in negative degrees
gives isomorphisms $\cald^{-i,p} \xrightarrow{\cong} \calb^{-i,p}=
0,$ for $i>0.$
\item $\cald^{i,0}\cong \calb^{i,0} \cong
\begin{cases}
\bbz(0) &, \text{ if } i = 0 \\
0 &, \text{ otherwise }.
\end{cases}$
\item For $p> 0$ one has an exact sequence
$$
0\to \cald^{0,p} \to \calb^{0,p} \to H^0(pt;\bbc)^\frS = \bbr \to
\cald^{1,p} \to \calb^{1,p} \to 0.
$$
\item For $i\neq 0,1$ one has $\cald^{i,p} \cong \calb^{i,p}.$
\end{enumerate}
Now,
$$
\calb^{0,p}\ \cong\
\begin{cases}
\bbz(2k) &, \text{ if } p = 2k \\
0         &, \text{ if } p \text{ is odd}
\end{cases}
$$
and
$$ \calb^{1,p}\ \cong\
\begin{cases}
\bbz/2\bbz &, \text{ if } p = 2k+1 \\
0         &, \text{ if } p \text{ is even}
\end{cases}.
$$

Hence, $\cald^{0,p} \cong 0,$ for all $p>0$, and for $k \geq 1$
one has short exact sequences:
\begin{equation}
\label{eq:l1}
0\to \bbz(2k) \to \bbr \to \cald^{1,2k} \to 0 \tag{$\rm e_k$}
\end{equation}
and
\begin{equation}
\label{eq:l2}
0\to \bbr \to \cald^{1,2k-1}  \to \bbz/2\bbz\to 0. \tag{$\rm o_k$}
\end{equation}
Using similar exact sequences, one concludes that
\begin{equation}
\cald^{1,2k-1}_\bbr = \bbr \hookrightarrow \cald_\bbc^{1,2k-1} = \bbc
\end{equation}
is the inclusion as fixed point set, and the change of coefficients homomorphism
is given by
\begin{align}
\label{eq:change}
\cald^{1,2k-1} = \bbr^\times & \longrightarrow \cald^{1,2k-1}_\bbr \\
x & \longmapsto \log{|x|}.\notag
\end{align}

It follows that $\cald^{1,2k} = \bbr/ \bbz(2k)$ and that
$\cald^{1,2k-1} \cong \bbr^\times.$
See Table \ref{table:dc-pt} for a display of the bigraded group structure of $\cald^{*,*}$.
\bigskip

\end{example}

\medskip

\subsection{Product structure I: positive weights}

The constructions in this section are modelled after
\cite{esnault}.

\begin{definition}
Given $0\leq \alpha \leq 1$ and $p,q \geq 0$, define a pairing of
complexes of presheaves
\begin{equation}
\cup_\alpha \colon \dAp^i \otimes \dAq^j \to \de{A}{p+q}^{i+j}
\end{equation}
by
$$
(a,\theta,\omega) \cup_\alpha (a',\theta',\omega') := (a\cdot a',\
\theta \wedge \theta',\ \omega\wedge \Xi_\alpha[a',\theta']
+(-1)^i \Xi_{1-\alpha}[a,\theta]),
$$
where $(a,\theta,\omega) \in \brAp^i \oplus F^p\cale^i \oplus
\cale^{i-1}$, $(a',\theta',\omega') \in \brAq^j \oplus F^q\cale^j
\oplus \cale^{j-1}$ and for $(u,\eta) \in \brAp^i \oplus
F^p\cale^i$ one defines
$$
\Xi_t ( u, \eta) := (1-t)\ \tau_p(u) +t\
\eta.
$$
\end{definition}

The following results are straightforward and their proofs are
left to the reader.

\begin{proposition}
\label{prop:basic}
Fix $p,q \geq 0$ and let $X$ be a proper real holomorphic
manifold.
\begin{enumerate}[i.]
\item The pairings $\cup_\alpha$ are homotopic to each other, for
all $0 \leq \alpha \leq 1,$ and they induce pairings
$$
\cup \colon \dcrA{i}{p}{X} \otimes \dcrA{j}{q}{X} \to
\dcrA{i+j}{p+q}{X},
$$
satisfying $a\cup b = (-1)^{ij} b\cup a.$
\item The natural maps $\dcrA{*}{p}{X}\to \bca{*}{p}{X}$, for
$p\geq 0$, are compatible with the respective multiplicative
structures of Deligne and Bredon cohomology.
\item The natural map $\iota \colon \dcrA{*}{0}{X} \to
\bca{*}{0}{X(\bbc)}$ is a ring isomorphism.
\end{enumerate}
\end{proposition}

Well-known facts from Bredon cohomology allow the computation of a
few more examples.

\begin{example}
\label{ex:cell}
Let $X \in \rav$ be a real \emph{cellular} proper algebraic
variety, and let $CH^*(X)$ denote its Chow ring, seen as a
bigraded ring where $CH^p(X)$ has degree $(2p,p)$. As a bigraded
ring, the Bredon cohomology of $X$ with $\uZ$-coefficients is
given by $\bcz{*}{*}{X(\bbc)} \cong CH^*(X)\otimes \calb^{*,*},$
cf. \cite{dS&lf-quadrics}. Since $CH^*(X)$ is free and Bredon cohomology
is a \emph{geometric cohomology theory} in the sense of \cite{Karpenko},
it follows that $\bca{*}{*}{X} \cong CH^*(X) \otimes \calb_A^{*,*},$ where
$\calb_A^{*,*}:= \bca{*}{*}{pt}$ and $A\subset \bbr$ is a subring
of $\bbr$. Furthermore, the singular cohomology of $X(\bbc)$ with
$\bbc$-coefficients is invariant and of Hodge type $(p,p)$. Hence,
the long exact sequence in the Proposition \ref{prop:change} and
the above observations give natural isomorphisms:
\begin{align}
\label{eq:isompp}
\dcrA{2p}{p}{X} & \cong \bca{2p}{p}{X(\bbc)} \cong
H^{2p}(X(\bbc);A(p))^\frS \\
& \cong CH^p(X)\otimes A.\notag
\end{align}

In particular, when $X=\bbp^p$, one has $\bca{*}{*}{\bbp^p}\cong
\calb_A^{*,*}[h]/ \la h^{p+1} \ra$, where $h \in
\bcz{2}{1}{\bbp^p}$ is the \emph{first Chern class} of the
hyperplane bundle in Bredon cohomology. Define $\xi \in
\dcrA{2}{1}{\bbp^p}$ as the element corresponding to $h$ via the
isomorphisms \eqref{eq:isompp}. It follows that
$\dcrA{2j}{j}{\bbp^p} \cong \bbz $ is generated by $\xi^j$, for
all $0\leq j \leq p$ and is $0$ for $j>p.$
\end{example}

Let $\xi \in \dcrA{2}{1}{X\times \bbp^p}$ denote the pull-back
under the projection $X\times \bbp^p \to \bbp^p$ of the class
$\xi\in \dcrA{2}{1}{\bbp^p}$ defined above, and let $\pi \colon
X\times \bbp^p \to X$ denote the projection onto $X$.

Given $\alpha \in \dcrA{i}{-r}{X\times \bbp^p},$ with $r, k \geq
0$, define $\alpha \cup \xi^k \in \dcrA{i+2k}{k-r}{X\times
\bbp^p}$ as follows:
\begin{equation}
\label{eq:ppp}
\alpha\cup \xi^k :=
\begin{cases}
\alpha \cdot h^k &, \text{ if } k\leq r \\
(\alpha \cdot h^r)\cup \xi^{k-r} &, \text{ if } k \geq r.
\end{cases}
\end{equation}
Here we are identifying Deligne and Bredon cohomology with weights
$\leq 0,$ and being careful to differentiate between $h$ and $\xi$
and between the product $\alpha\cdot h^k$ in Bredon cohomology and
the $\cup$ product for Deligne cohomology with positive weights.

The following result is a preliminary version of a more general
projective bundle formula.

\begin{proposition}
\label{prop:projbf1}
Let $X$ be a proper real holomorphic manifold. Given integers
$p\geq 0 $ and  $i, q \in \bbz$, the map
\begin{align*}
\psi \colon \oplus_{j=0}^p  \dcrA{i-2j}{q-j}{X} &\ \longrightarrow
\dcrA{i}{q}{X\times \bbp^p}\\
(a_0,\ldots,a_p) &\ \longmapsto \pi^*a_0 + \pi^*a_1 \cup \xi +
\cdots + \pi^* a_p \cup \xi^p
\end{align*}
is an isomorphism.
\end{proposition}
\begin{proof}
Denote by $\xi_\bbc$ the image of $\xi$ under the cycle map into
singular cohomology with complex coefficients. Since $\xi_\bbc$ is
of Hodge type $(p,p)$ and invariant under $\frS,$ the map
$a\mapsto \pi^*a\cup \xi_\bbc^j$ from $ H^{i-2j}(X(\bbc);\bbc) $
to $H^{i}(X(\bbc)\times \bbp^p(\bbc);\bbc)$ sends
$F^{q-j}H^{i-2j}(X(\bbc);\bbc)^\frS$ to $F^{q}H^{i}(X(\bbc)\times
\bbp^p(\bbc);\bbc)^\frS$.

Recall that Chern classes are compatible under the forgetful
functor from Bredon cohomology to singular cohomology and that the
projective bundle formula holds in both theories; cf. \cite{dS-plms}.
The result now follows from the definition of $\cup \xi$ in Deligne cohomology
and the five-lemma suitably applied to the long exact sequences in
Proposition \ref{prop:change}.
\end{proof}

\subsection{Product structure II: arbitrary weights}
\label{subsec:prodII}

Given $p,q \geq 0$ define the product
\begin{equation}
\label{eq:prod-neg}
\cup \colon \dcrA{i}{-p}{X} \otimes \dcrA{j}{-q}{X} \to
\dcrA{i+j}{-p-q}{X}
\end{equation}
simply as the product  in Bredon cohomology. Given $0\leq |r| \leq
p$, one uses \eqref{eq:ppp} to define $$ \star \ \colon
\dcrA{i}{-p}{X} \otimes \dcrA{j}{p-r}{X}\rightarrow
\dcrA{i+j+2p}{p-r}{X\times \bbp^p} $$ by sending $a\otimes b
\mapsto a\star b:= (a\cup \xi^p)\cup \pi^* b $. Now, let $
\pi_\dagger \colon \dcrA{i+j+2p}{p-r}{X\times \bbp^p} \to
\dcrA{i+j}{-r}{X}$ denote the composition of $\psi^{-1}$ (see
Proposition \ref{prop:projbf1}) with the projection onto the last
factor. Sending $a\otimes b\ \mapsto \  \pi_\dagger (a\star b)$
defines the product
\begin{equation}
\label{prod:mix}
\cup \colon \dcrA{i}{-p}{X} \otimes \dcrA{j}{p-r}{X} \to
\dcrA{i+j}{-r}{X},
\end{equation}

Combining Proposition \ref{prop:basic}, \eqref{eq:prod-neg} and
\eqref{prod:mix} one defines a product on Deligne cohomology which
makes $\dcrA{*}{*}{X} \to \bca{*}{*}{X(\bbc)}$ a natural map of
bigraded rings.

\begin{corollary}
\label{cor:pbf}
Under the hypothesis of Proposition \ref{prop:projbf1}, one has an
isomorphism of bigraded rings
$$
\dcrA{*}{*}{X}[T]/\la T^{p+1} \ra \cong \dcrA{*}{*}{X\times
\bbp^p}
$$
when  $T$ is given the bigrading $(2,1).$
\end{corollary}

The ring structure of $\cald^{*,*}$ can be
easily read from Table \ref{Table2}. See Example \ref{ex:br0-pt} for notation.
\bigskip

\begin{table}
\begin{tabular}{l!{\vrule}ccccccccc}
%
\multicolumn{1}{F}{8} &
%
\multicolumn{3}{V}{} &
%
\multicolumn{1}{V}{$0$} &
\multicolumn{1}{H}{$\bbr/\bbz(8)$} &
\multicolumn{1}{P}{$\tau^3\ve^2$} &
\multicolumn{1}{V}{$0$} &
\multicolumn{1}{P}{$\tau^2\ve^4$} \\
%
\multicolumn{1}{F}{7} &
%
\multicolumn{3}{V}{} &
%
\multicolumn{1}{V}{$0$} &
\multicolumn{1}{P}{$\rs$} &
\multicolumn{1}{V}{$0$} &
\multicolumn{1}{P}{$\tau^2\ve^3$} &
\multicolumn{1}{V}{$0$} \\
%
\multicolumn{1}{F}{6} &
%
\multicolumn{3}{V}{} &
%
\multicolumn{1}{V}{$0$} &
\multicolumn{1}{H}{$\bbr/\bbz(6)$} &
\multicolumn{1}{P}{$\tau^2\ve^2$} &
\multicolumn{1}{V}{$0$} &
\multicolumn{1}{P}{$\tau\ve^4$} \\
\multicolumn{1}{F}{5} &
%
\multicolumn{3}{V}{} &
%
\multicolumn{1}{V}{$0$} &
\multicolumn{1}{P}{$\rs$} &
\multicolumn{1}{V}{$0$} &
\multicolumn{1}{P}{$\tau\ve^3$} &
\multicolumn{1}{V}{$0$} \\
%
\multicolumn{1}{F}{4} &
\multicolumn{3}{V}{} &
\multicolumn{1}{V}{$0$} &
\multicolumn{1}{H}{$\bbr/\bbz(4)$ }&
\multicolumn{1}{P}{$\tau\ve^2$} &
\multicolumn{1}{V}{$0$} &
\multicolumn{1}{P}{$\ve^4$} \\
%
%
\multicolumn{1}{F}{3} &
\multicolumn{3}{V}{} &
\multicolumn{1}{V}{$0$} &
\multicolumn{1}{P}{$\rs$} &
\multicolumn{1}{V}{$0$} &
\multicolumn{1}{P}{$\ve^3$} &
\multicolumn{1}{V}{$0$} \\
\multicolumn{1}{F}{2} &
\multicolumn{3}{V}{} &
\multicolumn{1}{V}{$0$} &
\multicolumn{1}{H}{$\bbr/\bbz(2)$} &
\multicolumn{1}{P}{$\ve^2$} &
\multicolumn{1}{V}{$0$} &
\multicolumn{1}{V}{$0$}  \\
%
%
\multicolumn{1}{F}{1} &
\multicolumn{3}{V}{} &
\multicolumn{1}{V}{$0$} &
\multicolumn{1}{P}{$\rs$} &
\multicolumn{1}{V}{$0$} &
\multicolumn{1}{V}{$0$} &
\multicolumn{1}{V}{$0$}   \\
%
%
\multicolumn{1}{F}{0} &
\multicolumn{3}{V}{} &
\multicolumn{1}{H}{$1$} &
\multicolumn{1}{V}{$0$} &
\multicolumn{1}{V}{$0$} &
\multicolumn{1}{V}{$0$} &
\multicolumn{1}{V}{$0$}   \\
\hline
%
%
\multicolumn{1}{F}{-1} &
\multicolumn{3}{V}{} &
\multicolumn{1}{V}{$0$} &
\multicolumn{4}{V}{} &  \\
%
%
\multicolumn{1}{F}{-2} &
\multicolumn{3}{V}{} &
\multicolumn{1}{H}{$2\tau^{-1}$} &
\multicolumn{4}{V}{} & \\
%
%
\multicolumn{1}{F}{-3} &
\multicolumn{3}{V}{} &
\multicolumn{1}{P}{$\theta$} &
\multicolumn{4}{V}{} \\
%
%
\multicolumn{1}{F}{-4} &
\multicolumn{2}{V}{} &
\multicolumn{1}{P}{$\ve^{-1}\theta$} &
\multicolumn{1}{H}{$2\tau^{-2}$} &
\multicolumn{4}{V}{}   \\
%
%
\multicolumn{1}{F}{-5} &
\multicolumn{1}{V}{} &
\multicolumn{1}{P}{$\ve^{-2}\theta$} &
\multicolumn{1}{V}{$0$} &
\multicolumn{1}{P}{$\tau^{-1}\theta$} &
\multicolumn{4}{V}{} \\
%
\multicolumn{1}{F}{-6} &
\multicolumn{1}{P}{$\ve^{-3}\theta$} &
\multicolumn{1}{V}{$0$} &
\multicolumn{1}{P}{$\ve^{-1}\tau^{-1}\theta$} &
\multicolumn{1}{H}{$2\tau^{-3}$} &
\multicolumn{4}{V}{}  \\
%
%
\multicolumn{1}{F}{-7} &
\multicolumn{1}{V}{$0$} &
\multicolumn{1}{P}{$\ve^{-2}\tau^{-1}\theta$} &
\multicolumn{1}{V}{$0$} &
\multicolumn{1}{P}{$\tau^{-2}\theta$} &
\multicolumn{4}{V}{}  \\
%
\multicolumn{1}{F}{-8} &
\multicolumn{1}{P}{$\ve^{-3}\tau^{-1}\theta$} &
\multicolumn{1}{V}{$0$} &
\multicolumn{1}{P}{$\ve^{-1}\tau^{-2}\theta$} &
\multicolumn{1}{H}{$2\tau^{-4}$} &
\multicolumn{4}{V}{}  \\ \hline
%
%
%
\multicolumn{1}{F}{} &
{\color{blue} -3} &
{\color{blue} -2} &
{\color{blue} -1} &
{\color{blue} 0 } &
{\color{blue} 1 } &
{\color{blue} 2} &
{\color{blue} 3} &
{\color{blue} 4} &
\end{tabular}
\caption{Generators for cohomology of a point}
\label{Table2}
\end{table}

\newcommand{\bexp}{\eta}

\section{The exponential sequence}
\label{sec:exp}

In this section we show that $\de{\bbz}{1}$ is quasi-isomorphic to
$\calo^\times [-1]$. As
an immediate corollary, we obtain an \emph{exponential
  sequence} relating the cohomology of $\br{\bbz}{1}$,  $\calo$ and
$\calo^\times$. Using this sequence we obtain a new proof of
\emph{Weichold's Theorem} -- a classical result in real algebraic
geometry.

\begin{definition}Let $X$ be a smooth algebraic variety and let
  $(\calr_X^*,d_\calr)$ denote the following  cochain complex of
  abelian sheaves on $X$
\[
\cale^{0,0\, \times}_X \xrightarrow{\overline{\partial}\log}
\cale^{0,1}_X
\xrightarrow{\overline{\partial}}\cale_X^{0,2}\xrightarrow{\overline{\partial}}\dotsb
\]
where $\cale_X^{0,0\, \times}\subset\cale_X^0$ denotes the subsheaf of
nowhere zero functions and $\cale_X^{p,q}\subset\cale_X^{p+q}$ denotes
the subsheaf on invariant $(p+q)$-forms of Hodge type $(p,q)$.
\end{definition}

\begin{remark}
\begin{enumerate}
\item It is easy to check that $(\calr^*,d_\calr)$ is a  resolution of
  $\calo^\times$. Similarly, $\calo \to \cale^{0,0}
  \xrightarrow{\overline{\partial}} \cale^{0,1}
  \xrightarrow{\overline{\partial}}  \cdots $ is a soft
  resolution of $\calo$.
\item The exponential map $\calo \to \calo^\times$ extends to a map
$\exp \colon \cale^{0,*} \to \calr^*$ between the respective resolutions.
\end{enumerate}
\end{remark}

\begin{definition}
\label{def-delta}
Recall that an element in $\br{\bbz}{1}^1(U)$ is represented by
sums of pairs of the form $f\otimes m$,  with $f$ and $m$
equivariant maps such that
\begin{enumerate}
\item $f\colon S\to\bbc^\times$ is smooth and $\pi\colon S\to U$ is in
  $\widehat{U}$;
\item $m\colon S\to\bbz\in\underline{\bbz}(S)$ is locally constant.
\end{enumerate}
Let  $\bexp \colon \br{\bbz}{1}^1\to\cale^{0,0\, \times}$ be the map sending
 $f\otimes m\in \br{\bbz}{1}^1(U)$ to
\[
\bexp(f\otimes m)=\hat{\pi}_!(f^{m}).
\]
Alternatively,
$
\bexp(f\otimes m)(u)=\prod_{s\in\pi^{-1}(u)}f(s)^{m(s)}.
$
\end{definition}

\begin{remark}
It follows from this definition that, given $\sigma\in\br{\bbz}{1}^0$ and
$\alpha\in\br{\bbz}{1}^1$ we have
$\tau_1^0\sigma=\log\bexp(\sigma(1))-\log\bexp(\sigma(0))$ and $\tau_1^1\alpha= d\log \bexp\alpha$.
\end{remark}

\begin{proposition}
\label{exp-seq}
There is a quasi-isomorphism  $\varsigma\colon\de{\bbz}{1}\to\calr^*[-1]$
such that the composite
$\cale^*[-1]\to\de{\bbz}{1}\xrightarrow{\varsigma}\calr^*[-1]$ induces the map
$\exp\colon\bbc_X[-1]\to\calo_X^\times[-1]$, and the composition
$\cale^{0,*}[-1] \to \de{\bbz}{1} \to \calr^*[-1]$ induces the map
$\exp\colon\calo_X[-1]\to\calo_X^\times[-1]$.
\end{proposition}
\begin{proof}Let $X$ be a projective real variety and
let $\varsigma\colon \de{\bbz}{1}\to\calr^*[-1]$ denote the map of complexes
displayed in the following diagram:
$$
\xymatrix{
\ar[r] & \br{\bbz}{1}^0\ \ar[d]\ar[r]^-{d_{\de{\bbz}{1}}^{-1}}  &
\br{\bbz}{1}^1\oplus\cale^{1,0}\oplus\cale^0\ar[r]^-{d_{\de{\bbz}{1}}^0}
\ar[d]^{\bexp\cdot\exp}  &
\cale^{1,1}\oplus\cale^{2,0}\oplus\cale^1\ar[r]\ar[d]^{p^{0,1}} &\dotsb\\
 \ar[r] & 0        \ar[r]                    &
\calr^0   \ar[r]^{\overline{\partial}\log}               & \calr^1
\ar[r]^{\overline{\partial}}&  \dotsb
}
$$
where $\bexp\cdot\exp$ denotes the map $(f,\omega,h)\mapsto
\bexp(f)\exp h$, and  $p^{0,1}$ denotes the projection from $\cale^1$
to $\cale^{0,1}$. Similarly, for $i>1$ we set $\varsigma^i=p^{0,i}$,
where $p^{0,i}\colon\cale^i\to\cale^{0,i}$ is the projection.

It is easy to check  $\varsigma$ is a map of cochain complexes and
that  the cohomology presheaves of both complexes are concentrated in
dimension $1$.  Hence it suffices to check that $H^1(\varsigma)$ is an
isomorphism on the stalks.

\noindent
\emph{Surjectivity:} Note that $H^1(\calr^*[-1]) =
H^0(\calr^*)\cong{\calo^\times}$ and let  $g\in\calo^\times_u$, $u\in
X$. Let $h\in\calo_u$ be such that $\exp{h}=g$.  Set $\alpha=0$ in
$\bbz(1)_{\calb{r},u}$ (so that $\alpha$ can be represented, for
example, by $1\otimes 0$) and set $w=\frac{\partial
  g}{g}\in\cale^{1,0}_u$. Then we have,  $d_{\de{\bbz}{1}}^0(\alpha,
\omega, h)=0$ and $\varsigma_u(\alpha, \omega,  h)=g$.

\noindent
\emph{Injectivity:}  Let $(\alpha, \omega,
h)\in\bbz(1)^1_{\cald/\bbr,u}$, $u\in X$, be such that
$\varsigma_u(\alpha, \omega, h)=0$ and $d^0_{\de{\bbz}{1}}(\alpha,
\omega, h)=0$. The first equality just means that $-h$ is  a logarithm
for $\bexp(\alpha)$. From the second equality we get
\[
\omega=-dh-\tau^1_1\alpha=-(dh+d\log\bexp(\alpha))=0.
\]
Let $\sum_i f_i\otimes m_i$ be a representative for $\alpha$ (cf. Definition~\ref{exp-seq}).
Choose  $\log f_i$ so  that $\sum_i m_i\log f_i= \ - h $, shrinking neighborhoods if necessary,
and define
\[
\sigma:=\left( \Delta^1\ni t\mapsto \exp \left[\ -\  t \sum_i m_i \log f_i \right] \right )\in \bbz(1)_{\calb r, u}^0.
\]
Then, by our choice of $\log  f_i$, we have
$d^{-1}_{\de{\bbz}{1}}(\sigma)=(\alpha,\omega,h)$.

The remaining assertions are evident.
\end{proof}

\begin{corollary}
\label{cor:qi}
The complexes $\de{\bbz}{1}$ and $\calo^\times[-1]$ are quasi-isomorphic.
\end{corollary}

\begin{corollary}[Exponential Sequence]
\label{cor:exp-seq}
Let $X$ be a smooth proper real algebraic variety. Then there is a
long exact sequence
{\footnotesize
$$
\to
H_{\text{Br}}^{*,1}(X(\bbc);\uZ)\xrightarrow{\vartheta}H^*(X;\calo_X)
\xrightarrow{\exp} H^*(X;\calo_X^*) \to H_{\text{Br}}^{*+1,1}(X(\bbc);\uZ)\to
$$
}
where $\vartheta$ denotes the composite
$$
H_{\text{Br}}^{*,1}(X(\bbc);\uZ)\xrightarrow{\tau_1}
H^*(X(\bbc);\bbc)^\mathfrak{G}\twoheadrightarrow
H^{*}(X;\calo_X)=H_{\overline{\partial}}^{*,1}(X(\bbc))^\mathfrak{G},
$$
and the latter denotes the invariants of the Dolbeault cohomology
of the complex manifold $X(\bbc)$.
\end{corollary}
\begin{proof}
From the hypercohomology long exact sequence of the cone and
Lemma~\ref{exp-seq} we get the following exact sequence
\[
\to H_{\text{Br}}^{*,1}(X(\bbc);\uZ)\oplus
\check\bbh^*(X;F^1\cale^*_X)\xrightarrow{\iota_1} \check\bbh^*(X;\cale^*_X)
\to H^*(X;\calo_X^\times)\to
\]
which, in turn, gives
{\small
\[
H_{\text{Br}}^{*,1}(X(\bbc);\uZ)\xrightarrow{\tau_1} \coker
\left(\check\bbh^*(X;F^1\cale^*_X)\to  \check\bbh^*(X;\cale^*_X)\right)
\to  H^*(X;\calo_X^\times).
\]
}
Now, since $\cone(F^1\cale^*\to\cale^*)\simeq\cale^{0,*}$ and since
$\check\bbh^*(X;F^1\cale^*_X)\to  \check\bbh^*(X;\cale^*_X)$
is injective, we have
\[
\coker \left(\check\bbh^*(X;F^1\cale^*_X)\to  \check\bbh^*(X;\cale^*_X)\right) \cong \check\bbh^*(X;\cale_X^{0,*})\cong H^*(X;\calo_X).
\]
The assertion about the map $\vartheta$ follows immediately  from the
construction of this  sequence.
\end{proof}

\begin{remark}
A similar exact sequence appears in \cite{kras-char}.
\end{remark}

\subsection{An application}

Given a real variety $X$, denote $S= \pi_0(X(\bbr))$. Hence
$H^0(X(\bbr);\zs) \cong (\zs)^S$ and $\tilde{H}^0(X(\bbr);\zs) \cong
(\zs)^S/\zs $, where $\zs \subset (\zs)^S$ is the subgroup of constant
functions.

\begin{lemma}
\label{lem:br21}
\label{H^{2,1}-real curve}
Let $X$ be an irreducible, smooth, projective curve over $\bbr$, of genus
$g$. Let $c$ denote the number of connected components of $X(\bbr)$
Then
$$
H_{\text{Br}}^{2,1}(X(\bbc);\uZ)\ \cong\ \bbz \times (\zs)^S/\zs \
\cong \ \bbz\times
\left\{
\begin{array}{rl}
(\bbz/2)^{c-1} & \text{if $c\neq 0$}\\
0              & \text{if $c=0$}
\end{array}
\right\}
$$
\end{lemma}
\begin{proof}
By Poincar\'{e} dualtity $H_\text{Br}^{2,1}(X(\bbz);\uZ)\cong
H^{\frS}_0(X(\bbc);\uZ)$, where
the last group denotes $\bbz$-graded Bredon homology.
The following
exact sequence is well known (see \cite{LLFM03})
$$
0\to H^{\text{sing}}_0(X(\bbc)/\frS;\bbz)\to
H^{\text{Br}}_{0,0}(X(\bbc);\uZ)\to H_0(X(\bbr);\zs)\to 0.
$$

If $X(\bbr)=\varnothing$ the sequence gives
$H^\frS_0(X(\bbc);\uZ)\cong\bbz$ since under the above assumptions
$X(\bbc)/\frS$ is connected.

If $X(\bbr)\neq\varnothing$ then $H^{\text{Br}}_{0,0}(X(\bbc);\uZ)\cong
\bbz\times \widetilde{H}^{\text{Br}}_{0,0}(X(\bbc);\uZ)$, where the
$\widetilde{H}^{\text{Br}}_{*,*}(-;\uZ)$ denotes reduced Bredon homology.
There is a  reduced version of the sequence above which gives $
\widetilde{H}^{\text{Br}}_{0,0}(X(\bbc);\uZ)\cong
\widetilde{H}_0(X(\bbr);\zs) \cong (\zs)^{c-1}, $ because
$\widetilde{H}^\text{sing}_0(X(\bbc)/\frS;\bbz)=0$.
\end{proof}

Denote $V := H^1(X_\bbc,\calo)$ and let $\Lambda \subset V$ be the
lattice
$$
\Lambda := \operatorname{Im}\{ j_\bbc \colon
H^1_{\text{sing}}(X_\bbc,\bbz(1) ) \to H^1(X_\bbc,\calo) \},
$$
so that $Pic(X_\bbc)\cong V/\Lambda$.
Taking fixed points gives a short exact sequence
\begin{equation}
\label{eq:ses-pic}
0 \to \Lambda^\frS \to V^\frS \to Pic_0(X_\bbc)^\frS \to
H^1(\frS,\Lambda) .
\end{equation}
This shows that $V^\frS/\Lambda^\frS$ is the connected component $Pic_0(X_\bbc)^\frS_0$.
Now, Proposition \ref{prop:bor-bre} shows that
$\bca{1}{1}{X(\bbc)} \cong \bcabor{1}{1}{X(\bbc)}$, hence one can use  the Leray-Serre spectral sequence to conclude that the image of the natural map $j\colon \bcz{1}{1}{X(\bbc)} \to H^1(X_\bbc,\calo)$
is precisely $\Lambda^\frS$.

\begin{proposition}
\label{prop:pic}
Let $X$ be an irreducible, smooth, projective curve over $\bbr$, of genus
$g$.  Then
$$
Pic(X) \cong Pic_0(X_\bbc)_0^\frS \times \bcz{2}{1}{X(\bbc)}.
$$
\end{proposition}
\begin{proof}
By the previous results, $Pic(X)$ fits in the following exact sequence
$$
\bcz{1}{1}{X(\bbc)} \xrightarrow{j} H^1(X_\bbc;\calo)^\frS\to
Pic(X) \xrightarrow{c_1} \bcz{2}{1}{X(\bbc)} \to 0.
$$
The result follows.
\end{proof}

\begin{corollary}[Weichold's Theorem~{\cite[Proposition 1.1]{PW91}}]
\label{cor:weich}
With $X$ as above, let $c$ denote the number of connected components
of $X(\bbr)$.
Then
$$
Pic(X) \ \cong\ \bbz \times (\bbr/\bbz)^g\times
\left\{
\begin{array}{rl}
(\bbz/2)^{c-1} & \text{if $c\neq 0$}\\
0              & \text{if $c=0$}
\end{array}
\right\}
$$
\end{corollary}
\begin{proof}
This follows directly from the proposition together with Lemma
\ref{lem:br21}.
\end{proof}

\newcommand{\hatrho}{\hat{\rho}}
\newcommand{\hrz}{\hatrho_{\alpha_0}}
\newcommand{\fz}{{f^0_{\alpha_0}}}
\newcommand{\hGG}{\hat{G}_{\alpha_0\alpha_1}}

\section{The group $\dcr{2}{2}{X}$}
\label{sec:h22}

In this section we derive a geometric interpretation of the
integral cohomology group $\dcr{2}{2}{X}$ for a real projective
variety $X$. As a motivation we start with a geometric interpretation of the
Bredon cohomology group $\bcz{2}{2}{Y}$ of an arbitrary $\frS$-manifold $Y.$

\subsection{Variations on the theme of {$ H_{\rm Br}^{2,2}(Y,\mathbb{Z})$} }

As a $\frS$-manifold, the sphere $S^{2,2}$ is isomorphic to
$\bbp^1(\bbc)$ with the action induced by a linear involution
$\sigma$. Describing the action in homogeneous coordinates by
$\sigma \colon [x_0:x_1]\mapsto [x_0:-x_1]$, one sees that
$SP_\infty(\bbp^1(\bbc)) \equiv \bbp^\infty(\bbc)$ inherits the
linear involution $\sigma( [x_0:x_1:x_2:\cdots ]) = [x_0 : -x_1
:x_2:-x_3:\cdots ]$ and that this is an equivariant
$K(\uZ,(2,2))$; cf. \cite{dS-DT}. In particular, one can define
a linear isomorphism of tautological bundles  $\tau \colon
\sigma^*\calo(-1) \to \calo(-1)$ satisfying $\tau \circ
(\sigma^*\tau) = 1$.

Given a $\frS$-space $Y$, with involution $\sigma$, let $P_1(Y)$
denote the set of pairs $(L,\tau)$ where:
\begin{enumerate}[{\bf t1)}]
\item $L$ is a (smooth) complex line bundle on Y;
\item $\tau \colon \sigma^*L \to L$ is a bundle isomorphism;
\item $\tau\circ (\sigma^* \tau) = 1$.
\end{enumerate}
Define an equivalence relation on $P_1(L)$\ by\ $(L,\tau) \sim_1
(L',\tau')$ iff there exists an isomorphism $\phi \colon L \to L'$
such that $\phi \circ \tau = \tau' \circ \sigma^*\phi$.

\begin{lemma}
\label{lem:L1}
The tensor product of line bundles induces a group structure on the set
$\call_1(Y):= P_1(Y)/\!\sim_1$ of equivalence classes of pairs
satisfying  {t1)--t3)}. Furthermore, this group is
naturally isomorphic to $\bcz 2 2 Y$.
\end{lemma}
\begin{proof}
The first assertion is clear and the last one follows from the fact
that $(\bbc\bbp^\infty, \sigma)$  with the linear involution $\sigma$
described above is an equivariant $K(\uZ,(2,2))$.
\end{proof}

Recall that a \emph{Real vector bundle} $(E,\tau)$ on a
$\frS$-manifold $(Y,\sigma)$ consists of a complex vector bundle $E$
on $Y$ together with an isomorphism $\tau \colon \overline{\sigma^* E}
\to E$ satisfying $\tau \circ {\overline{\sigma^*\tau}} = Id$.
Now, consider the set $P_2(Y)$ consisting of pairs
$(L,\mbq)$ satisfying:
\begin{enumerate}[{\bf p1)}]
\item $L$ is a (smooth) complex line bundle on $Y$;
\item $\mbq \colon L\otimes \overline{\sigma^*L} \to \bone_Y$ is
  an isomorphism of Real line bundles, where $L\otimes
  \overline{\sigma^*L}$ carries the tautological Real line bundle structure;
\end{enumerate}
Denote $(L,\mbq)\sim_2 (L',\mbq')$ iff there is an isomorphism $\phi \colon L
\to L'$ satisfying $\mbq'\circ(\phi \otimes \overline{\sigma^*\phi}) =
\mbq$ and observe that this is an equivalence relation on
$P_2(Y)$.
\begin{lemma}
\label{lem:L2}
The tensor product also induces a group structure on the set
$\call_2(Y):= P_2(Y)/\! \sim_2$ of isomorphism classes of pairs
  $(L,\mbq)$ satisfying {\bf  p1)--p2)}.
\end{lemma}

Finally, consider the complex of sheaves $G^0\xrightarrow{a} G^1$ on
$\gman$ where
\begin{equation}
\label{eq:G0}
G^0(U)  = \{ f \colon U \to \bbc^\times \mid f \text{ is smooth }   \}
\end{equation}
and
\begin{equation}
\label{eq:G1}
G^1(U)  = \{ f \colon U \to \bbc^\times \mid f \text{ is smooth and
  equivariant}   \} ,
\end{equation}
where $a \colon G^0 \to G^1$ is the ``transfer map'' $a(f) = f
\cdot \overline{\sigma^*f}.$

\begin{proposition}
\label{prop:h22br}
There are natural isomorphisms
$$
\bcz{2}{2}{Y} \cong H^{2,2}_\text{bor}(Y,\uZ) \cong \call_1(Y) \cong
\call_2(Y) \cong \bbh^1(Y_{\text{eq}}; G^0\to G^1).
$$
\end{proposition}
\begin{proof}
The first two isomorphisms follow from Lemma \ref{lem:L1} and
Proposition \ref{prop:bor-bre}, respectively, and the last
isomorphism is a tautology.

Given $(L,\tau) \in P_1(Y)$, pick a hermitian metric $h \colon
L\otimes \overline{L} \to \bone_Y$ on $L$ and define $\mbq_\tau^h \colon
L\otimes \overline{\sigma^* L} \to \bone_Y $ as the composition
$L\otimes \overline{\sigma^* L} \xrightarrow{1\otimes \overline{\tau}}
L\otimes \overline{L} \to \bone_Y ,
$
where $\overline{\tau} \colon \overline{\sigma^* L} \to \overline{L}$
is the map induced by $\tau$. It is easy to see that $\mbq_\tau^h$ is
an isomorphism of Real line bundles, and hence we obtain an element
$(L,\mbq_\tau^h) \in P_2(Y)$.

Suppose that $\psi \colon L' \to L$ induces an equivalence
$(L',\tau') \sim_1 (L,\tau)$ and pick a hermitian metric $h$ on
$L$. One sees that $(L',\mbq_{\tau'}^{\psi^* h})
\sim_2 (L,\mbq_\tau^h)$. It follows that
one has a well-defined homomorphism $\call_1(Y) \to \call_2(Y)$
sending $[L,\tau]$ to $\la L, \mbq_\tau^h \ra$, where $h$ is any
choice of metric on $L$. The construction of the inverse homomorphism is evident.
\end{proof}

Let $S=\pi_0(Y^\frS)$ denote the set of connected components of
the fixed point set $Y^\frS$, and identify $H^0(Y^\frS;\bbz^\times) \equiv
(\bbz^\times)^S$. Observe that the K\"unneth formula yields a
natural isomorphism $H^{2,2}_\text{bor}(Y^\frS;\uZ)\cong
H^2(Y^\frS \times B\frS;\bbz(2)) \cong H^0(Y^\frS;\bbz^\times)
\oplus H^2_{\text{sing}}(Y^\frS;\bbz(2))$. Using the first
identification in Proposition \ref{prop:h22br}  one considers the
composition
$$\bcz{2}{2}{Y} \equiv
H^{2,2}_\text{bor}(Y;\uZ)\to H^{2,2}_\text{bor}(Y^\frS;\uZ)
\to H^0(Y^\frS;\bbz^\times)
$$
of the restriction map followed by the evident projection to
obtain a natural homomorphism
\begin{equation}
\label{eq:sig}
 \sig \colon \bcz{2}{2}{Y} \to ({\bbz^\times})^S .
\end{equation}

This map has a natural geometric interpretation when one uses the
identification $\bcz{2}{2}{Y}\equiv \call_2(Y)$. Given $\la L, \mbq
\ra \in \call_2(Y)$, the restriction of $\mbq$ to $L|_{Y^\frS}$\
becomes a non-degenerate hermitian pairing, and hence it has a
well-defined signature $\sig_{\la L, \mbq \ra} \in
(\bbz^\times)^S$. It is easy to see that this is another description
of \eqref{eq:sig}.

\begin{definition}
\label{def:sig}
We call  $\sig \colon \bcz{2}{2}{Y} \to ({\bbz^\times})^S$
the \emph{equivariant signature map} of $Y$. The image
$\sigtor(Y)\subseteq (\bbz^\times)^S$ of the torsion subgroup
$\bcz{2}{2}{Y}_\text{tor}$ under $\sig$ is called the
\emph{equivariant signature group} of $Y.$ In the case where
$Y=X(\bbc)$ for a real algebraic variety $X$ with $S =\pi_0( X(\bbr)
)$, we denote the equivariant signature group of $X(\bbc)$ simply by
$\sigtor(X)$.
\end{definition}

\begin{example}
\label{ex:cur22}
When $X$ is a projective algebraic curve, it follows from the cohomology
sequence of the pair $(E\frS\times_\frS X(\bbc), B\frS \times
X(\bbr))$  that $\sig$ is an isomorphism. As a consequence, one
obtains isomorphisms $\bcz{2}{2}{X(\bbc)}\cong Br(X) \cong
\sigtor(X),$ where $Br(X)$ is the Brauer group of $X$,  since
$Br(X)\cong (\bbz^\times)^S$ when $X$ is an algebraic curve;
cf. \cite{Witt}.
\end{example}

\subsection{The Deligne group $\dcr{2}{2}{X}$}

The Hodge filtration on singular cohomology induces a filtration
on Bredon cohomology, with $$F^j\bcz{n}{p}{X(\bbc)} :=
\varphi^{-1}\jmath^{-1}F^jH^n_{\text{sing}}(X(\bbc),\bbc);$$ see
diagram \eqref{eq:forget} for notation.

\begin{proposition}
\label{prop:aleph}
For any smooth real projective variety $X,$ one has
$$
F^2\bcz{2}{2}{X(\bbc)}=\bcz{2}{2}{X(\bbc)}_{tor}=
\operatorname{im}(\varrho),$$ where $\varrho$ is the cycle map
from Deligne to Bredon cohomology \eqref{eq:forget}. In particular
the image of the composition $$\Psi \colon \dcr{2}{2}{X}
\xrightarrow{\varrho} \bcz{2}{2}{X(\bbc)} \xrightarrow{\sig}
H^0(X(\bbr),\bbz^\times)$$ is the equivariant signature group
$\sigtor(X).$
\end{proposition}
\begin{proof}
Consider diagram \eqref{eq:forget} with $p=n=2$. Since the middle
row is exact, one concludes that $\operatorname{im}{(\varrho)} =
F^2\bcz{2}{2}{X},$ by definition. On the other hand, $\jmath \circ
\varphi \colon \bcz{2}{2}{X(\bbc)} \to H^2(X(\bbc);\bbc)$ factors
as $ \bcz{2}{2}{X(\bbc)} \to
\bcz{2}{2}{X(\bbc)}/\operatorname{tor} \hookrightarrow
\bcz{2}{2}{X(\bbc)}\otimes \bbq \hookrightarrow
H^2_{\operatorname{sing}}(X(\bbc);\bbz(2))\otimes \bbq
\hookrightarrow H^2(X(\bbc);\bbc). $  The injectivity of
$\bcz{2}{2}{X(\bbc)}\otimes \bbq \hookrightarrow
H^2_{\operatorname{sing}}(X(\bbc);\bbz(2))\otimes \bbq $ follows
from the isomorphism $\bcz{2}{2}{X(\bbc)} \cong H^{2,2}_{\rm
bor}(X(\bbc);\uZ)$ and well-known facts in equivariant cohomology.
Since $$H^2_{\operatorname{sing}}(X(\bbc);\bbz(2))\otimes \bbq
\cap F^2H^2(X(\bbc);\bbc) = 0,$$ one concludes that
$$F^2\bcz{2}{2}{X(\bbc)} =
\bcz{2}{2}{X(\bbc)}_{\operatorname{tor}}.$$
\end{proof}

The next goal is to provide a geometric interpretation to the
kernel of the surjection $\Psi \colon \dcr{2}{2}{X} \to
\sigtor(X)$. To this purpose, consider triples $(L,\nabla,\mbq)$
where $L$ is a holomorphic line bundle over $X(\bbc)$, $\nabla$ is
a holomorphic connection on $L$ and $\mbq \colon L\otimes
\overline{\sigma^*L}\to \bone $ is a holomorphic isomorphism of
Real line bundles satisfying the following properties:

\begin{enumerate}
\item The restriction of ${\mbq}$ to $X(\bbr)$ is a
positive-definite hermitian metric on $L_{|X(\bbr)}.$

\item As a section of $(L\otimes \overline{\sigma^*L})^\vee $,
$\mbq$ is parallel with respect to the connection induced by
$\nabla.$
\end{enumerate}

A morphism between two such triples $f \colon (L,\nabla,\mbq)\to
(L',\nabla',\mbq')$ consists of a line bundle map $f \colon L \to
L'$ such that ${\mbq}'\circ(f\otimes \overline{\sigma^*f})
= {\mbq}$ and $\nabla'\circ f = (1\otimes f)\circ \nabla.$

\begin{definition}
\label{def:rlbc}
Given a real variety $X,$ let $\rlbc{X}$ denote the set of
isomorphism classes $\la L,\nabla,\mbq\ra$ of triples  as above.
This is a group under the operation
$$
\la L,\nabla,\mbq\ra\odot \la L',\nabla',\mbq'\ra := \la L\otimes
L', \nabla\otimes 1 + 1 \otimes \nabla', \mbq\cdot \mbq'\ra,
$$
which we call \emph{the differential Picard-Witt group} of
$X$.
\end{definition}
\bigskip

The group $\rlbc{X}$ has an alternative definition in terms of a
complex $\calp^* : \calp^0 \xrightarrow{D} \calp^1 \xrightarrow{D}
\calp^2$ of presheaves on $\rav$. Given a real analytic variety
$U$, define
$$
\calp^0(U) := \calo^\times_{\bbc} (U) = \{ f \colon U \to \bbc^\times
\mid f \text{ is holomorphic } \};
$$
and
$$
\calp^1(U) := \Omega^1_{\bbc}(U) \oplus \calo^\times_{\bbr_+}(U) ,
$$
where $\Omega^1_\bbc(U)$ denotes the holomorphic $1$-forms on $U$
and $\calo^\times_{\bbr_+}(U)$ denotes the subgroup of
$\calo^\times_\bbc(U)$ consisting of those holomorphic Real
functions $f$ (i.e. $\overline{\sigma^* f} = f$) which are
positive on the \emph{real locus} $U(\bbr):= U^\frS$. Finally,
define $\calp^2(U)$ as the group of Real holomorphic $1$-forms on
$U$, in other words
$$
\calp^2(U) := \Omega^1_\bbr(U) := \{ \psi \in \Omega^1_\bbc(U) \mid
\overline{\sigma^* \psi} = \psi \}.
$$
Define $D \colon \calp^0(U)\to \calp^1(U)$ as $D(g) := ( {dg}/{g},\
g\cdot \overline{\sigma^* g})$ and
$D \colon \calp^1(U)\to \calp^2(U)$ as $D(\psi,f) := \psi +
\overline{\sigma^* \psi} - {df}/{f}.$

\begin{proposition}
\label{prop:rlbc}
$\rlbc{X}$ is naturally isomorphic to
$\check{\bbh}^1(X_{eq};\calp^*)$.
\end{proposition}
\begin{proof}
This is straightforward.
\end{proof}

\begin{remark}
\label{rem:cocycle1}
Given $\la L, \nabla, \mbq \ra$ in $\rlbc{X}$ one can always find an
equivariant cover $\underline{\calu}$ and a cocycle $\mbc$ in
$\check{C}^1(\underline{\calu},\calp^*)$ representing $\la L, \nabla,
\mbq \ra$ that has the form
$\mbc = ( (g_{\alpha_0\alpha_1}) \ ; \ \left( \ (\psi_{\alpha_0}),\ (
1 ) \ \right) )$. In other words, $[\mbc]$ is determined by
$(g_{\alpha_0\alpha_1}), (\psi_{\alpha_0})$ satisfying:
\begin{description}
\item[i] $g_{\alpha_0\alpha_1} \in \Omega^0_\bbc(U_{\alpha_0\alpha_1})$ and $\delta
  (g_{\alpha_0\alpha_1}) = 1$ (\emph{gives the cocycle condition for a holomorphic
  line bundle $L$});
\item[ii] $\left( \frac{dg_{\alpha_0\alpha_1}}{g_{\alpha_0\alpha_1} }
  \right) = \delta ( \psi_{\alpha_0} )$ (\emph{gives the holomorphic connection on $L$});
\item[iii] $g_{\alpha_0\alpha_1} \cdot
  \overline{\sigma^*g_{\alpha_0\alpha_1}} = 1 $ (\emph{gives the
  hermitian form $\mbq$});
\item[iv] $\psi_{\alpha_0} \in \Omega^1_\bbc(U)$ and
  $\psi_{\alpha_0} + \overline{\sigma^* \psi_{\alpha_0}} = 0$,
  i.e. $\psi_{\alpha_0}$ is a holomorphic anti-invariant $1$-form
  (\emph{$\mbq$ is parallel}).
\end{description}
\end{remark}
The main result of this section is the following.

\begin{theorem}
\label{thm:h22}
If $X$ is a smooth real projective variety then  one has a natural
short exact sequence
$$
0\to \rlbc{X} \to \dcr{2}{2}{X} \xrightarrow{\Psi} \sigtor(X) \to 0.
$$
\end{theorem}
\begin{proof}
See Appendix \ref{app:proof}.
\end{proof}

\section{A remark on number fields}
\label{sec:ex}

Let $F$ be a number field and let $\Gamma_\bbr $ and $\Gamma_\bbc$
denote the sets of real  and complex embeddings of $F$,
respectively. One can write $\Gamma_\bbc = \Gamma^+_\bbc \times \frS$,
where $\Gamma_\bbc^+$ contains one chosen element in each $\frS$-orbit
of $\Gamma_\bbc$.

Abusing language, write $\dcrA{r}{p}{F}$ instead of
$\dcrA{r}{p}{X}$, where $X := Spec(F\otimes_\bbq \bbr)$. Observe that
$X(\bbc) \equiv \Gamma_\bbr \coprod \Gamma_\bbc$, and hence
\begin{align}
\label{eq:isos}
\dcrA{r}{p}{F} & =  \dcrA{r}{p}{\Gamma_\bbr} \times
\dcrA{r}{p}{\Gamma_\bbc}   \notag \\ &\equiv
\dcrA{r}{p}{*}^{\Gamma_\bbr} \times
\dccA{r}{p}{*}^{\Gamma_\bbc^+} \\ & \cong \dcrA{r}{p}{*}^s \times
\dccA{r}{p}{*}^{t}, \notag
\end{align}
where $A$ is a subring of $\bbr$ and  $\Gamma_\bbr = \{
\varphi_1,\ldots,\varphi_s\}$\ and\ $ \Gamma^+_\bbc = \{
\eta_1,\ldots,\eta_t \}$. In particular,
\begin{align}
\label{eq:H1}
\dcr{1}{1}{F}\ &\equiv\ (\bbr^\times)^{\Gamma_\bbr} \times
(\bbc^\times)^{\Gamma_\bbc^+}\ \cong\ (\bbr^\times)^s \times (\cs)^t
\quad
\text{ and }  \\
\dcrr{1}{1}{F}\ &\equiv\ \bbr^{\Gamma_\bbr} \times
\bbr^{\Gamma_\bbc^+}\ \cong\ \bbr^s \times \bbr^t ; \notag
\end{align}
cf. Example \ref{ex:dc-point}.
Taking adjoints to the evaluation maps $F^\times \times \Gamma_\bbr
\to \bbr^\times$ and $F^\times \times \Gamma_\bbc^+ \to \bbc^\times$
gives a monomorphism
\begin{equation}
\label{eq:reg1}
F^\times \to (\bbr^\times)^{\Gamma_\bbr} \times
(\bbc^\times)^{\Gamma_\bbc^+}\equiv \dcr{1}{1}{F} .
\end{equation}
Since $\oplus_{p\geq 0} \dcr{p}{p}{F}$ is a graded commutative
ring this map induces a homomorphism
\begin{equation}
\rho \colon T(F^\times) \to \oplus_{p\geq 0} \dcr{p}{p}{F},
\end{equation}
where $T(F^\times)$ is the tensor algebra of $F^\times$.
Using the commutativity of the diagram
$$
\xymatrix{
F^\times \otimes F^\times \ar[r]^-{\rho} & \dcr{1}{1}{F} \otimes
\dcc{1}{1}{F} \ar[r]^-{\cup} \ar[d]_-{\varrho\otimes \varrho} &
\dcr{2}{2}{F} \ar[d]_{\varrho}^{\cong} \\
& \bcz{1}{1}{F} \otimes \bcz{1}{1}{F} \ar[r]_-{\cdot} &
\bcz{2}{2}{F},
}
$$
together with the description of the ring structure of the Bredon
cohomology of a point \ref{ex:br0-pt}, one concludes that if $a
\neq 0, 1$, then $\varrho(a\otimes (1-a)) = 0$. It follows that $\varrho$
descends to a homomorphism
\begin{equation}
\label{eq:reg2}
\bar\varrho \colon K^M_*(F) \to \oplus_{p\geq 0} \dcr{p}{p}{F},
\end{equation}
from the Milnor $K$-theory ring of $F$ to the ``diagonal'' subring of
the integral Deligne cohomology of $F$.

\begin{remark}
\label{rem:fin}
\begin{enumerate}[i.]
\item If follows from the work of Bass and Tate that
$$
K^M_*(\bbr)/2K^M_{\geq 2}(\bbr) \cong\ \oplus_i\ \cald^{i,i}
\ \ \text{and} \ \ K^M_*(\bbr)/2K^M_{*}(\bbr) \cong \bbz/2[\varepsilon]\cong \oplus_i \calb^{i,i}.
$$
\item Since $\pi \colon Spec(F\otimes_\bbq \bbr) \to Spec(\bbr)$ is a
  finite cover, one has an   additive transfer   homomorphism
\begin{align}
\label{eq:transfer}
\pi_! \colon \dcrr{1}{1}{F} & \to \dcrr{1}{1}{\bbr}\\
(x_1,\ldots,x_s;y_1,\ldots,y_t) & \mapsto x_1+\cdots + x_s + 2y_1 +
\cdots + 2y_t; \notag
\end{align}
see \eqref{eq:H1}.
\item In subsequent work we will show that the homomorphism \eqref{eq:reg2} is a
  particular case of a natural transformation between the motivic
  cohomology of a real variety and its integral Deligne cohomology.
\end{enumerate}
\end{remark}

It follows from \eqref{eq:H1} and \eqref{eq:change} that the composition
$$ F^\times  \to \dcr{1}{1}{F}  \to
H^{1}_{\cald/\bbr}(F;\bbr(1))
$$
is given by
\begin{align}
F^\times & \longrightarrow \bbr^s \times \bbr^t \\
 x & \longmapsto
 (\log{|\varphi_1(x)|},\ldots,\log{|\varphi_s(x)|};\log{|\eta_1(x)|},\ldots,\log{|\eta_t(x)|}).
 \notag
\end{align}
Basic class field theory shows that the image of the units
$\mathfrak{o}^\times_F$ of the rings of integers of $F$ under this
map is a lattice $L$ in the hyperplane $H : x_1+\cdots + x_s +
2y_2 + \cdots + 2 y_t =0$, i.e., the kernel of the transfer
homomorphism \eqref{eq:transfer}. Therefore, the Euclidean
volume of this lattice in $H$ is given by $ \operatorname{Vol}(L)
= \frac{\sqrt{s+4t}}{2^t} \ R, $ where $R$ is the classical
regulator of $F$.



\appendix
\section{The Borel/Esnault-Viehweg version}
\label{sec:appEVv}

Given any equivariant cohomology theory $\mathfrak{h}^*$ on
$\frS$-spaces, one can define its corresponding \emph{Borel
version} $\mathfrak{h}^*_\text{bor}$ as
$\mathfrak{h}^*_\text{bor}(U):= \mathfrak{h}^*(U\times E\frS),$
where $E\frS$ is a contractible $\frS-CW$-complex on which $\frS$ acts freely.
In particular, one has an associated Borel cohomology theory
$\bcabor{n}{p}{X}$. This theory is $(0,2)$ periodic and can be
more easily calculated than Bredon cohomology, via Leray-Serre
spectral sequences
$$
E^{r,s}(p):= H^r(\frS,\mathcal{H}^s(X;A(p))) \Rightarrow
\bcabor{r+s}{p}{X(\bbc)}.
$$

It is easy to see that the natural map $\bbz_0(S^{p,p}) \to
F(E\frS,\bbz_0(S^{p,p}))$ is an equivariant homotopy equivalence,
thus giving the following result.
\begin{proposition}
\label{prop:bor-bre}
Let $Y$ be a $\frS$-space. For all $p\geq 0$ and $n\leq p$ one has
a natural isomorphism $\bca{n}{p}{Y} \cong \bcabor{n}{p}{Y}$.
\end{proposition}

In order to translate this construction into our context, denote
$X:= (\spec{\bbc})_{/\bbr}$ and let $E_\bullet \frS :=
{\caln}(X\to \spec{\bbr})$ be the nerve of the cover $X\to
\spec{\bbr}.$ This is a smooth simplicial real projective variety
with the property that $E_\bullet \frS(\bbc)$ is a simplicial
object in $\rav$ whose geometric realization
$|E_\bullet\frS(\bbc)|$ is a model for $E\frS$. In particular,
$\calz (E_\bullet \frS (\bbc) )$ defines a simplicial abelian
presheaf on $\gman$ whose associated complex (graded in negative
degrees) is denoted $\calz(E\frS)^*$

\begin{definition}
\label{def:borel}
Given a complex of presheaves $\calf^*$ on $\gman,$ define its
\emph{associated Borel complex} as
$$
\calf^*_\bor := \uhom{\calz(E\frS)^*}{\calf^*},
$$
and let $\iota_{\calf^*} \colon \calf^* \to \calf^*_\bor$ denote
the natural map induced by the projection $E_\bullet \frS \to
\spec{\bbr}.$ In particular, given $p\in \bbz$ one can define the
\emph{Borel version of the Deligne complex} as $\dApbor:=
\left(\dAp \right)_\bor$ and the Borel version of the Bredon
complex as $\brApbor$. Correspondingly, given $X\in \rav$ define
its \emph{Borel version of Deligne cohomology} as
$$
\dcrAbor{i}{p}{X} := \check\bbh^i( X_{\text{eq}}; \dApbor),
$$
and the Borel version of Bredon cohomology as
$$
\bcabor{i}{p}{X} := \check\bbh^i( X_{\text{eq}}; \brApbor)
$$
\end{definition}
It follows from the definitions that $$\dApbor =
\cone\left(\brApbor \oplus F^p\cale^*_\bor \to
\cale^*_\bor\right)[-1] .$$

\begin{proposition}
\label{prop:triang}
Given $p\in \bbz$ one has map of  exact triangles on $\rav:$
$$
\xymatrix{ \dAp \ar[r]\ar[d] & \brAp \oplus F^p\cale^*
\ar[r]\ar[d] & \cale^* \ar[d] \ar[r] & \dAp[1] \ar[d] \\
\dApbor \ar[r] & \brApbor \oplus F^p\cale_\bor^* \ar[r] &
\cale_\bor^* \ar[r] & \dApbor[1]
 }
$$
\end{proposition}
\begin{corollary}
Let $X$ be a real analytic manifold such that $X^\frS =
\emptyset.$ Then $\iota \colon \dcrA{i}{p}{X}\to
\dcrAbor{i}{p}{X}$ is an isomorphism for all $i$ and $p.$
\end{corollary}
\begin{proof}
Since multiplication by $2$ is invertible in $\cale^*$ and
preserves the filtration $\{ F^p\cale^* \}$ one concludes that
$\iota \colon F^p\cale^*\to F^p\cale^*_\bor$ is a
quasi-isomorphism for all $p\in \bbz.$ The result now follows from
the five-lemma and the fact the same result holds for Bredon
cohomology.
\end{proof}

\begin{remark}
When $X$ is a projective smooth real variety, the \emph{associated
Borel version of the Deligne cohomology groups}
$\dcrAbor{i}{p}{X}$ defined above coincide with the Deligne
cohomology for real varieties introduced by Esnault and Viehweg in \cite{esnault}. However,
in general, the groups $\dcrA{i}{p}{X}$ and $\dcrAbor{i}{p}{X}$ are
rather distinct.
\end{remark}

\section{Proof of Theorem \ref{thm:h22}}
\label{app:proof}

We will need the following two technical lemmas.

\begin{lemma}
\label{lem:hom}
Let $\imath \colon  \uZ \to \cale^0$ denote the natural inclusion
of sheaves on $\gman$. Given any real number $\lambda$ one can
find a map of presheaves $\xi_{\lambda} \colon \uZ \to
\br{\bbz}{2}^0 $ such that for each $U\in \gman$, the composition
$$
\uZ ( U ) \xrightarrow{\xi_\lambda} \br{\bbz}{2}^0 (U)
\xrightarrow{\tau}  \cale^0(U)
$$
coincides with $\lambda \cdot \imath$.
\end{lemma}
\begin{proof}
Let $a>0$ be a positive real number and let $I_a$ denote the interval
$[a,1]$ if $a<1$ and $[1,a]$ if $a\geq 1$. Let $\phi_{a,i} \colon \Delta^2
\to \bbr^\times \times \bbr^\times \subset \bbc^\times \times
\bbc^\times$,  $i=1,2$ be smooth maps that give an oriented
triangulation of the rectangle $I_a\times I_a \subset \bbr^\times
\times \bbr^\times.$ It follows that
$$ \int_{\Delta^2}
(\phi_{a,1}^*\omega_2 + \phi_{a,2}^*\omega_2) = (\log{a})^2.
$$

Let $p_U \colon U \to *$ denote the projection to the point, where $U$
is a $\frS$-manifold. Given $\nu \in \uZ(U)$, define
$\xi_{\lambda}(\nu) \in \br{\bbz}{2}^0(U) $ by
$$\xi_{\lambda}(\nu) :=
\begin{cases}
(p_U^*\phi_{a,1} + p_U^*\phi_{a,2})\otimes \nu
&, \text{ if } \lambda >0 \text{ and } a =
\exp{\sqrt{\lambda}} \\
- (p_U^*\phi_{a,1} + p_U^*\phi_{a,2})\otimes \nu &, \text{ if }
  \lambda < 0 \text{ and } a =
\exp{\sqrt{|\lambda|}}. \end{cases}
$$
It is clear that $\xi_\lambda$ is a homomorphism satisfying the
desired conditions.
\end{proof}

\begin{lemma}[The period argument]
\label{lem:period}
Given $\alpha \in \brZp(U)^0$ satisfying $d^B\alpha =0$, then $\tau(\alpha) \in
\underline{\bbz(p)}(U)$. In other words, $\tau(\alpha) \in \cale^0(U)$
is a locally constant equivariant function with values in
$\bbz(p)$.
\end{lemma}
\begin{proof}
The cocycle $\alpha$ is represented by an element of the form $(a,\sum_i
f_i\otimes \nu_i) $, where $f_i \colon S_i \times \Delta^p \to
(\bbc^\times)^p$ is smooth and equivariant, $p_i \colon S_i \to U$
is an equivariant covering map, and $\nu_i \colon S_i \to \bbz$ is
equivariant and locally constant. Since $d\tau(\alpha) = \tau (d^B \alpha) =
0$, we know that $\tau(\alpha) $ is an equivariant locally constant
function.

Given $x_0\in U$ and $y\in p_i^{-1}(x_0) \subset U_i$ the
restriction of $f_i$ to $y\times \Delta^p$ is a smooth proper map,
and hence one obtains a smooth integral $p$-simplex $f_{i \#} [y\times
\Delta^p ]$ in $(\bbc^\times)^p$ with boundary $\partial f_{i \#}
[y\times \Delta^p ]= f_{i \#}[y\times \partial \Delta^p].$  It
follows that
$$
T_{\alpha,x_0} := \sum_i\ \sum_{y\in \pi_i^{-1}(x_0)}\ \nu_i(y) f_{i
\#}[y\times \Delta^p]
$$
is a smooth integral $p$-cycle on $(\bbc^\times)^p$ and a
simple inspection shows that $\int_{T_{\alpha,x_0}}
\omega_p = \tau(\alpha) (x_0)$. On the other hand, since $T_{\alpha,x_0}$
represents an integral homology class in $(\bbc^\times)^p$, then
$\int_{T_{\alpha,x_0}} \omega_p$ is a period of $\omega_p$ over an
integral homology class and hence it lies in $\bbz(p)$.
\end{proof}

\subsection{Cocycles for Bredon cohomology}

As a preparation to the main arguments of next section, we describe
the isomorphism
\begin{equation}
\label{eq:explicit}
\Phi \colon \check\bbh^2(X(\bbc)_\text{eq}; \br{\bbz}{2}) \to
\check\bbh^1(X(\bbc)_\text{eq};G^0\to G^1)
\end{equation}
in terms of $\check{\text{C}}$ech cocycles; cf. Proposition \ref{prop:h22br}.

If $U$ is a $\frS$-manifold, we denote by $U^{\rm triv}$ the same
space with the trivial $\frS$-action. Given integers $n \geq j \geq 0$
let $D^{n,j} \subset (n-j)\cdot \bone \oplus j\cdot \xi$ denote the
unit ball in $\bbr^n$ with the action induced by the
representation. We say that a $\frS$-manifold $U$ has {\bf Type $\frS$} if
it is equivariantly isomorphic to $(D^{n,0})^{\rm triv} \times
\frS $. We say that $U$ has {\bf Type $P$} if $U\cong D^{n,j}, $ for
some $j\geq 0$.

Let $Y$ be a $\frS$-manifold of dimension $n$. A \emph{good cover}
for $Y$ is an open cover $\underline{\mathcal{V}} = \{ V_\alpha \mid \alpha
\in \Lambda \}$ such that all non-empty intersections are contractible. We may even assume that these intersections are homeomorphic to disks. We say that $\underline{\mathcal{V}}$ is \emph{equivariantly good} if the group permutes the open sets in the cover. A cover with these properties always yields an {\bf equivariant good cover} $\underline{\calu}$, i.e. a cover by $\frS$-invariant open sets having the property
that all elements $U_{\alpha_0\cdots \alpha_k} := U_{\alpha_0}\cap
\cdots \cap U_{\alpha_k}$ in the nerve of $\underline{\calu}$ have
either Type $\frS$ or Type $P$. Furthermore, if {\bf one} of the
elements $\alpha_0,\ldots, \alpha_k$ has Type $\frS$, then the
intersection $\alpha_0\cdots \alpha_k$ is required to have Type $\frS$. Also, if {\bf
all} $\alpha_0,\ldots,\alpha_k$ have Type $P$, then
intersection $\alpha_0\cdots \alpha_k$ must also have Type $P$. Abusing language, we
say that the index $\alpha_0\cdots \alpha_k$ has Type $\frS$ or Type $P$, accordingly.

\begin{remark}
\label{rem:cofinal}
Using totally convex balls for a Riemannian metric
so that $\sigma$ acts via isometries, one sees that any
$\frS$-manifold $Y$ has  an equivariant good cover. Also, these covers
form a cofinal family amongst the family of all equivariant covers of
$Y$.
\end{remark}

\begin{lemma}[The local obstruction argument]
\label{lem:loa}
Let $Y$ be a $\frS$-manifold whose path-components are all
contractible. Then, for all $p\geq 0$,  the complex $\brZp(Y\times
\frS)$ is acyclic, i.e.  $H^j\left(\brZp(Y\times
\frS)\right) = 0$ for all $j\neq 0$. In particular, $\brZp(U)$ is
acyclic if $U$ is a
$\frS$-manifold of Type $\frS$ and $\brZp(U\times \frS)$ is acyclic if $U$
has Type $P$.
\end{lemma}
\begin{proof}
Using the isomorphism $U\times \frS \cong U^{\rm triv} \times
\frS$ in Remark \ref{rem:cofinal}(ii), pick one point in each path
component of $Y$  and obtain an equivariant strong deformation
retraction $Y\times \frS \simeq \pi_0(Y) \times \frS$, where
$\pi_0(Y)$ is given the discrete topology. It follows from
Proposition \ref{rem:simp-ps} that one has a quasi-isomorphism
$\brZp(Y\times \frS) \simeq \brZp(\pi_0(Y)\times \frS)$. Now, the
cohomology of the latter complex gives  the bigraded Bredon
cohomology groups $\bcz{*}{p}{\pi_0(Y)\times\frS} \cong H^*_{\rm
sing}(\pi_0(Y); \bbz(p))$. The result follows.
\end{proof}

Let $\underline{\calu} := \{ U_\alpha \mid \lambda \in \Lambda \}$ be
an ``equivariant good cover" of $Y$ and let $\mbh = \left(
(h^i_{\alpha_0\cdots\alpha_j}) \mid i+j = 2, j\geq 0 \right) $ be a
$\check{\text{C}}$ech cocycle representing an element $[\mbh] \in
\bbh^2(Y_\text{eq};\br{\bbz}{2}).$
The cocycle condition gives:
\begin{equation}
\label{eq:diff}
d^B h^2_{\alpha_0} = 0 \quad \text{and} \quad \delta h^i_{\alpha_0
  \cdots \alpha_j} = (-1)^j d^B h^{i-1}_{\alpha_0 \cdots
  \alpha_{j+1}}  \ \ \text{for all}\ \ i \leq 1,
\end{equation}
where  $d^B$ and $\delta$ are the differentials in the Bredon and
$\check{\text{C}}$ech complexes, respectively.
\medskip

\noindent{\bf MAIN GOAL:}{\it
We will find a representative $((g_{\alpha_0\alpha_1})\ , \ (\rho_{\alpha_0}) ) \in \check{C}(\underline{\calu}, G^0 \to G^1)$ for $\Phi([\bbh])$, satisfying:
\begin{align*}
\delta(g_{\alpha_0\alpha_1}) & = 1 ; \\
  \left( g_{\alpha_0\alpha_1}\cdot \overline{\sigma^*g_{\alpha_1\alpha_1}} \right) & = \delta(\rho_{\alpha_0}); \\ (\overline{\sigma^* \rho_{\alpha_0}}) & = (\rho_{\alpha_0}).
\end{align*}
}

\begin{remark}
\label{rem:explanation}
Let $[\mbh]$ be as above. Given a fixed point $x_0 \in Y^{\frS}$, let
$U_{\alpha_0}$ be an element of the cover (necessarily of Type $P$)
containing $x_0$. Since $\brZp$ has homotopy invariant cohomology
presheaves, one has  natural isomorphisms
$H^j(\br{\bbz}{2}(U_{\alpha_0}) \cong \calb^{j,2}$. It follows from
\eqref{eq:diff} that  $h_{\alpha_0}^2$ represents
a class in  $H^2\left( \br{\bbz}{2}(U_{\alpha_0}) \right) \ \cong\
\calb^{2,2}\cong \bbz^\times $, thus giving an element in
$\bbz^\times$. Notice that this element is the same for any point $x
\in U_{\alpha_0} \cap Y^{\frS}$ and it is easy to see that it depends
only on the class $[\mbh] \in \bcz{2}{2}{Y^{\frS}}.$ Hence, the resulting
map $S \to \bbz^\times$ depends only on $[\mbh]$, and this
is an additional description of the signature map in terms of
$\check{\text{C}}$ech-Bredon cocycles.
\end{remark}

The identity
$\delta(h^0_{\alpha_0\alpha_1\alpha_2}) = (d^B
h^{-1}_{\alpha_0\cdots \alpha_3})$ gives
$\delta(\tau h^0_{\alpha_0\alpha_1\alpha_2} ) = 0$ and hence, since
$\cale^0$ is a soft sheaf, one can find $f^0_{\alpha_0\alpha_1} \in
\cale^0(U_{\alpha_0\alpha_1})$ such that
\begin{equation}
\label{eq:nu}
(\tau
h^0_{\alpha_0\alpha_1\alpha_2}) = \delta (f^0_{\alpha_0\alpha_1}).
\end{equation}

Let $p \colon Y\times \frS \to Y$ denote the projection. Given any
$\alpha_0\cdots \alpha_j$ in the nerve of the covering
$\underline{\calu}$, we use the same notation $p \colon
U_{\alpha_0\cdots \alpha_j} \times \frS \to U_{\alpha_0\cdots
  \alpha_j}$ to denote the corresponding projection. For any
presheaf $\calf$ on $\gman$ and $h \in \calf(U_{\alpha_0\cdots
  \alpha_j})$ let $p^*h \in \calf(U_{\alpha_0\cdots \alpha_j}
\times \frS)$ denote the pull-back of $h$ under $p$.
\medskip

\noindent{\bf \underline{Type $\frS$ case:}}\ \hfill
\smallskip

Recall that if  $\alpha_k$ is of Type $\frS$, for some $k=1,\ldots,j$, then
so is $U_{\alpha_0\cdots \alpha_j}$.

\noindent{\bf Step 1:}\  If $\alpha_0$ is of Type $\frS$, then the \emph{local
obstruction argument} (Lemma \ref{lem:loa}) implies that
one can find $t^1_{\alpha_0} \in \br{\bbz}{2}^1(U_{\alpha_0})$ such that
\begin{equation}
\label{eq:t1I}
d^B t^1_{\alpha_0} = h^2_{\alpha_0}.
\end{equation}

\noindent{\bf Step 2:}\ Assume that both $\alpha_0$ and $\alpha_1$ are
of Type $\frS$.  Since
\begin{equation}
\label{eq:t0I}
d^B\left( \ ({h}^1_{\alpha_0\alpha_1}) - \delta(t^1_{\alpha_0} )\
\right) = \delta(h^2_{\alpha_0}) - \delta(h^2_{\alpha_0}) = 0,
\end{equation}
by the cocycle condition, the local obstruction argument guarantees
the existence of $t^0_{\alpha_0\alpha_1} \in
\br{\bbz}{2}(U_{\alpha_1\alpha_1})$ such that
\begin{equation}
d^Bt^0_{\alpha_1\alpha_1} = ({h}^1_{\alpha_0\alpha_1}) -
\delta(t^1_{\alpha_0} ).
\end{equation}

\noindent{\bf Step 3:}\ Here we only consider $\alpha_0\cdots
\alpha_j$ in the nerve of the cover $\underline{\calu}$ for which all
$\alpha_k$'s are of Type $\frS$. Observe that
\begin{align*}
d^B\left( (h^0_{\alpha_0\alpha_1\alpha_2}) +
\delta(t^0_{\alpha_0\alpha_1}) \right) &=
- \delta ( h^1_{\alpha_0\alpha_1}  ) +  \delta ( d^B
t^0_{\alpha_0\alpha_1} ) \\
& = - \delta ( h^1_{\alpha_0\alpha_1}  ) +  \delta (
(h^1_{\alpha_0\alpha_1}) - \delta (t^1_{\alpha_0}) ) =
0.
\end{align*}
Applying $\tau$ one obtains
$d \left( (\tau h^0_{\alpha_0\alpha_1\alpha_2}) + \delta(\tau
t^0_{\alpha_0\alpha_1}) \right)
=0$, and hence Lemma \ref{lem:period} (the \emph{periods argument}) shows that
$$(\nu_{\alpha_0\alpha_1\alpha_2}):= (\tau
h^0_{\alpha_0\alpha_1\alpha_2}) + \delta(\tau
t^0_{\alpha_0\alpha_1})$$
consists of locally constant equivariant functions
$\nu_{\alpha_0\alpha_1\alpha_2} \colon U_{\alpha_0\alpha_1\alpha_2}
\to \bbz(2)$. Using \eqref{eq:nu} one writes
\begin{equation}
\label{eq:nudelta}
(\nu_{\alpha_0\alpha_1\alpha_2})= \delta \left( f^0_{\alpha_0\alpha_1}
+ \tau t^0_{\alpha_0\alpha_1}  \right).
\end{equation}

\noindent{\bf Step 4:}\ Denote $\hat{g}_{\alpha_0\alpha_1}:=
f^0_{\alpha_0\alpha_1} + \tau t^0_{\alpha_0\alpha_1}$ and
choose a square root $i$ of $-1$. Define
\begin{equation}
\label{eq:gI}
g_{\alpha_0\alpha_1} = \exp\left( \frac{1}{2\pi i}
\hat{g}_{\alpha_0\alpha_1} \right) \quad \text{and} \quad
\rho_{\alpha_0}=1.
\end{equation}
The cocycle condition
\begin{equation}
\label{eq:cocycleI}
g_{\alpha_0\alpha_1} g_{\alpha_1\alpha_2} g_{\alpha_2\alpha_0} = 1
\end{equation}
when all $\alpha$'s are of Type $\frS$ follows from
\eqref{eq:nudelta}. Since $\hat{g}_{\alpha_0\alpha_1} -
\overline{\sigma^* \hat{g}_{\alpha_0\alpha_1}} = 0$ one concludes that
\begin{equation}
\label{eq:trivI}
 {g}_{\alpha_0\alpha_1}\cdot \overline{\sigma^*
 {g}_{\alpha_0\alpha_1}} = 1 = \rho_\beta/\rho_\alpha
\end{equation}
on $U_{\alpha_0\alpha_1}$ and by definition
\begin{equation}
\label{eq:invI}
\overline{\sigma^* \rho_\alpha} = \rho_\alpha.
\end{equation}
\medskip

\noindent{\bf \underline{Type $P$ case:}}\ \hfill
\smallskip

Recall that if  $\alpha_k$ is of Type $P$, for all $k=1,\ldots,j$, then
so is $U_{\alpha_0\cdots \alpha_j}$.

\noindent{\bf Step 1:}\  If $\alpha_0$ is of Type $P$, then the {local
obstruction argument} (Lemma \ref{lem:loa}) gives some
$\hat{t}^1_{\alpha_0} \in  \br{\bbz}{2}^1(U_{\alpha_0}\times \frS)$
such that $d^B \hat{t}^1_{\alpha_0} = p^* h^2_{\alpha_0}$.

Given any $\alpha_0$ define
\begin{equation}
\label{eq:T}
T^1_{\alpha_0} =
\begin{cases}
p^*t^1_{\alpha_0} &, \text{ if } \alpha_0 \text{ is of Type $\frS$} \\
\hat{t}^1_{\alpha_0} &, \text{ if } \alpha_0 \text{ is of Type $P$},
\end{cases}
\end{equation}
and observe that $d^B T^1_{\alpha_0} = p^* h^2_{\alpha_0}$.

\noindent{\bf Step 2:}\ If either $\alpha_0$ or $\alpha_1$ is of Type $P$ then
\begin{equation}
d^B\left( \ (p^*{h}^1_{\alpha_0\alpha_1}) - \delta(T^1_{\alpha_0}) \
\right) = 0 .
\end{equation}
Since $\calb^{1,2}=0$, it follows from the local obstruction argument
(see Remark \ref{rem:explanation}) that one obtains
$\hat{t}^0_{\alpha_0\alpha_1} \in
\br{\bbz}{2}(U_{\alpha_0\alpha_1}\times \frS )$ such that
\begin{equation}
\label{eq:t0hat}
d^B\hat{t}^0_{\alpha_0\alpha_1} = (p^*{h}^1_{\alpha_0\alpha_1}) -
\delta(T^1_{\alpha_0}) .
\end{equation}

\noindent{\bf Step 3:}\ Here we consider those $\alpha_0\cdots
\alpha_j$ in the nerve of the cover $\underline{\calu}$ for which one
of the $\alpha_k$'s is of Type $P$. In this case one has $$d^B\left(
(p^*h^0_{\alpha_0\alpha_1\alpha_2}) +
\delta(\hat{t}^0_{\alpha_0\alpha_1}) \right) =0.$$
Applying $\tau$ one
obtains
$d \left( (\tau p^* h^0_{\alpha_0\alpha_1\alpha_2}) + \delta(\tau
\hat{t}^0_{\alpha_0\alpha_1}) \right)
=0$.
It follows from Lemma \ref{lem:period} that
$(\nu_{\alpha_0\alpha_1\alpha_2}) := (\tau p^*
h^0_{\alpha_0\alpha_1\alpha_2}) + \delta(\tau
\hat{t}^0_{\alpha_0\alpha_1})$ consists of equivariant locally constant functions
$U_{\alpha_0\alpha_1\alpha_2}\times \frS \to \bbz(2)$. In view of Remark \ref{rem:eq-vs-nonequ} we can also consider $\nu_{\alpha_0\alpha_1\alpha_2}$ as a non-equivariant locally constant function from $U_{\alpha_0\alpha_1\alpha_2}$ to $\bbz(2)$.

\noindent{\bf Step 4:}\ Let $\imath \colon U\times \frS \to U\times \frS$ be as in Remark~\ref{rem:eq-vs-nonequ}.
Observe that
\begin{equation}
d^B\left( \imath^*T^1_{\alpha_0}  -T^1_{\alpha_0}   \right) = \imath^* d^BT^1_{\alpha_0} - d^B T^1_{\alpha_0} = h^2_{\alpha_0} \imath^*p^*h^2_{\alpha_0} - h^2_{\alpha_0} = 0.
\end{equation}
It follows from the \emph{local obstruction argument} that
\begin{equation}
\label{eq:dBcl}
\imath^*T^1_{\alpha_0} - T^1_{\alpha_0} = d^B \tilde{\gamma}^0_{\alpha_0},
\end{equation}
for some  $\tilde{\gamma}^0_{\alpha_0} \in   \br{\bbz}{2}^0(U_{\alpha_0}\times \frS)$. In particular, one has $d^B\left( \tilde{\gamma}^0_{\alpha_0} + \imath^*\tilde{\gamma}^0_{\alpha_0}\right) =0$, and the \emph{period argument} shows that
\begin{equation}
\label{eq:per2}
\tau \left( \tilde{\gamma}^0_{\alpha_0}   \right)
+
\imath^*\tau \left( \tilde{\gamma}^0_{\alpha_0}   \right)
= \tilde{C}_{\alpha_0},
\end{equation}
for some equivariant locally constant function $\tilde{C}_{\alpha_0}$ on $U_{\alpha_0}\times \frS$ with values in $\bbz(2)$.

Define $\tilde{g}_{\alpha_0\alpha_1}:= p^*f^0_{\alpha_0\alpha_1} + \tau
\hat{t}^0_{\alpha_0\alpha_1}$.
Then
\begin{align}
\label{eq:tildeg}
( \tilde{g}_{\alpha_0\alpha_1} )-\imath^* (\tilde{g}_{\alpha_0\alpha_1}) - \delta (\tau(\tilde{\gamma}^0_{\alpha_0}))
 & =
\tau \hat{t}^0_{\alpha_0\alpha_1}   -
\imath^*  \tau \hat{t}^0_{\alpha_0\alpha_1} - \delta (\tau(\tilde{\gamma}^0_{\alpha_0}))   \\
& =
\tau \left( \hat{t}^0_{\alpha_0\alpha_1}   -
\imath^*\hat{t}^0_{\alpha_0\alpha_1}
- \delta (\tilde{\gamma}^0_{\alpha_0})  \right) . \notag
\end{align}

On the other hand, equations \eqref{eq:t0hat} and \eqref{eq:dBcl}
give
\begin{align}
d^B \left( \hat{t}^0_{\alpha_0\alpha_1}   -
\imath^*\hat{t}^0_{\alpha_0\alpha_1}
- \delta (\tilde{\gamma}^0_{\alpha_0}) \right)  &=
 \delta\left( - T^1_{\alpha_0} + \imath^*T^1_{\alpha_0}   - d^B (\tilde{\gamma}^0_{\alpha_0}) \right) \\
 & = 0,\notag
\end{align}
and hence Lemma \ref{lem:period} shows that
$( \tilde{g}_{\alpha_0\alpha_1} )-\imath^* (\tilde{g}_{\alpha_0\alpha_1}) - \delta (\tau(\tilde{\gamma}^0_{\alpha_0}))
$
consists of equivariant locally constant functions from $U_{\alpha_0\alpha_1}\times \frS $ to $\bbz(2)$.

Using the notation in Remark \ref{rem:eq-vs-nonequ}, define $\hat{g}_{\alpha_0\alpha_1} = F(\tilde{g}_{\alpha_0\alpha_1}) = f^0_{\alpha_0\alpha_1} + F(\tau(\hat{t}^0_{\alpha_0\alpha_1} ))$, and
$\gamma^0_{\alpha_0}:= F\tau(\tilde{\gamma}^0_{\alpha_0})$. With the same choice of square root $i$ of $-1$ as in the Type $\frS$ case, define
\begin{equation}
g_{\alpha_0\alpha_1} = \exp\left( \frac{1}{2\pi i}
\hat{g}_{\alpha_0\alpha_1} \right) \quad \text{and} \quad
\rho_{\alpha_0}=\exp\left( \frac{1}{2\pi i} \gamma^0_{\alpha_0}
\right).
\end{equation}
Step $3$ shows that
\begin{equation}
\label{eq:cocycleII}
g_{\alpha_0\alpha_1} g_{\alpha_1\alpha_2} g_{\alpha_2\alpha_0} = 1,
\end{equation}
and \eqref{eq:tildeg} together with subsequent remarks and properties of the functor $F$ show that \begin{equation}
\label{eq:trivII}
 {g}_{\alpha_0\alpha_1}\cdot \overline{\sigma^*
 {g}_{\alpha_0\alpha_1}} = \rho_{\alpha_1}/\rho_{\alpha_0}
\end{equation}
on $U_{\alpha_0\alpha_1}$.
Finally, \eqref{eq:per2} and Remark \ref{rem:eq-vs-nonequ} give
$ F(\tau(\tilde{\gamma}^0_{\alpha_0})) + \overline{\sigma^*F(\tau(\tilde{\gamma}^0_{\alpha_0}))} = F(\tilde{C}_{\alpha_0}):= C_{\alpha_0}$. In particular, $C_{\alpha_0}$ is an equivariant locally constant function on $U_{\alpha_0}$ with values in $\bbz(2)$, and hence
\begin{align*}
\left(\frac{1}{2\pi i} \gamma^0_{\alpha_0}\right) - \overline{\sigma^*\left( \frac{1}{2\pi i} \gamma^0_{\alpha_0}\right)} & = \frac{1}{2\pi i} \left(F(\tau(\tilde{\gamma}^0_{\alpha_0})) + \overline{\sigma^*F(\tau(\tilde{\gamma}^0_{\alpha_0}))}\right) \\
& = \frac{1}{2\pi i} C_{\alpha_0} \ \ \in \underline{\bbz(1)}(U_{\alpha_0}).
\end{align*}
This gives
\begin{equation}
\label{eq:invII}
\overline{\sigma^* \rho_{\alpha_0}} = \rho_{\alpha_0}.
\end{equation}
\medskip
In other words $(
(g_{\alpha_0\alpha_1}), (\rho_{\alpha_0}))$ defines a $1$-cocycle
in $\check{C}^*(\underline{\calu},G^0\to G^1)$. It is easy to
check that this process would send a $\check{\text{C}}$ech
coboundary to a coboundary in
$\check{C}^*(\underline{\calu},G^0\to G^1)$, realizing the isomorphism
\eqref{eq:explicit}.

\subsection{The proof}

We now prove Theorem \ref{thm:h22}.

\begin{proof}
Let $\underline{\calu}$ be an equivariant good cover of $X(\bbc)$ and
fix a cocycle $\mbc = (c^{r,s})_{r+s=2} \in \check{C}^2(\underline{\calu};\de{\bbz}{2})$ representing a class in
$\dcr{2}{2}{X}$. Since $c^{r,s} = 0$ for $r>2$, the cocycle condition
is given by
\begin{equation}
Dc^{2,0} = 0 \quad \text{and} \quad Dc^{r-1,s+1} = (-1)^s \delta
c^{r,s}, r\leq 1.
\end{equation}
Recall that $\de{\bbz}{2} = \operatorname{Cone}\left(
\br{\bbz}{2}\oplus F^2\cale^* \xrightarrow{\iota_2} \cale^*
\right)[-1]$.

Write
\begin{align}
c^{2,0} & = \left(\ (h^2_{\alpha_0}),\ (\omega^2_{\alpha_0}),\
(\theta^1_{\alpha_0})\ \right) \\
c^{1,1} & = \left( \  (h^1_{\alpha_0\alpha_1}),\ 0,\
(\theta^0_{\alpha_0\alpha_1})\ \right) \\
c^{0,2} & = \left( \  (h^0_{\alpha_0\alpha_1\alpha_2}),\ 0,\
0\ \right) \\
c^{-i,2+i} & = \left( \ (h^{-i}_{\alpha_0\ldots\alpha_{2+i}}),\ 0,\ 0\
\right), \
\ i\geq 1.
\end{align}

Since $Dc^{2,0}=0$, one has for all $\alpha_0 \in \Lambda:$
\begin{align}
d^Bh^2_{\alpha_0} & = 0 , \text{ (by definition) }; \\
d\omega^2_{\alpha_0} & = 0 ;\\
\tau(h^2_{\alpha_0}) - \omega^2_{\alpha_0} + d \theta^1_{\alpha_0}
\label{eq:triple}
&= 0.
\end{align}

Similarly, $Dc^{1,1} = \delta c^{2,0}$ gives the identities:
\begin{align}
(d^Bh^1_{\alpha_0\alpha_1}) & = \delta(h^2_{\alpha_0}) \\
0 & = \delta(\omega_{\alpha_0}) \\
- ( \tau h^1_{\alpha_0\alpha_1}) - (d\theta^0_{\alpha_0\alpha_1})&  =
\delta(\theta^1_{\alpha_0}),
\end{align}
$Dc^{0,2} = - \delta c^{1,1}$ gives:
\begin{align}
(d^Bh^0_{\alpha_0\alpha_1\alpha_2}) & = - \delta(h^1_{\alpha_0\alpha_1}) \\
(\tau h^0_{\alpha_0\alpha_1\alpha_2}) &=
\delta(\theta^0_{\alpha_0\alpha_1}) \label{eq:theta0}
\end{align}
and for all $r\leq 0$ one has :
\begin{align}
(d^B h^{-r-1}_{\alpha_0\ldots \alpha_{3+r}} ) & = (-1)^{r}
  \delta(h^{-r}_{\alpha_0\ldots \alpha_{r+2}}).
\end{align}

The assignment $\mbc:=(c^{r,s})\mapsto \mbh:= (h^i_{\alpha_0\cdots
  \alpha_j}) $
gives the cycle map from Deligne to Bredon cohomology. Assume that
$\mbc \in \ker{\Psi}$, hence $\Psi ([\mbc] ) = \sig([\mbh]) = 0$.
Therefore the $\check{\text{C}}$ech-Bredon cocycle $\mbh$ is
``unobstructed'' and we can apply the arguments in {\sc Type $\frS$
Case} above;  see \eqref{eq:t1I}. Furthermore, the
data in the $\check{\text{C}}$ech-Deligne complex gives a natural
choice for the $f^0_{\alpha_0\alpha_1}$ introduced in
\eqref{eq:nu}. More precisely, one can choose
$f^0_{\alpha_0\alpha_1} := \theta^0_{\alpha_0\alpha_1}$; cf.
\eqref{eq:theta0}.

It follows that one can take $g_{\alpha_0\alpha_1} := \exp\left(
\frac{1}{2\pi i} \hat{g}_{\alpha_0\alpha_1}  \right) $, with
\begin{equation}
\label{eq:ghat3}
\hat{g}_{\alpha_0\alpha_1} = \theta^0_{\alpha_0\alpha_1} +
\tau(t^0_{\alpha_0\alpha_1}),
\end{equation}
to obtain a cocycle for the line bundle associated to $[\mbh]$;
cf. \eqref{eq:gI}.
However, one can find an equivalent holomorphic cocycle as follows.

Write the $1$-form $a_{\alpha_0} := \theta^1_{\alpha_0} + \tau
t^1_{\alpha_0} = a^{1,0}_{\alpha_0} + a^{0,1}_{\alpha_0}$ as a sum of
their $(1,0)$ and $(0,1)$ parts, respectively, where
$t^1_{\alpha_0\alpha_1}$ is introduced in \eqref{eq:t1I}. Since $d
a_{\alpha_0} = d\left( \theta^1_{\alpha_0} + \tau t^1_{\alpha_0}
\right) = d\theta^1_{\alpha} + d \tau(h^2_{\alpha_0}) =
\omega^2_{\alpha_0}$, cf. \eqref{eq:t1I} and \eqref{eq:triple}, and
$\omega^2_{\alpha_0}$ is a form of type $(2,0)$ one concludes that
\begin{align}
\bar{\partial} a^{0,1}_{\alpha_0} & = 0 \label{eq:dbar1} \\
{\partial} a^{0,1}_{\alpha_0} & = - \bar{\partial} a^{1,0}_{\alpha_0}
\label{eq:dbar2} \\
\partial a^{1,0}_{\alpha_0} & = \omega^2_{\alpha_0}
\end{align}
It follows from the $\bar\partial$-Poincar\'e lemma that one can find
$f^0_{\alpha_0} \in \cale^0(U_{\alpha_0})$ (equivariant) such that
\begin{equation}
\label{eq:dbarsol}
\bar\partial f^0_{\alpha_0} = a^{0,1}_{\alpha_0}.
\end{equation}
Now define
\begin{equation}
\label{eq:gtilde}
(\tilde{g}_{\alpha_0\alpha_1}) := (\hat{g}_{\alpha_0\alpha_1}) +
\delta( f^0_{\alpha_0} ).
\end{equation}
Hence,
\begin{align*}
\bar\partial \tilde{g}_{\alpha_0\alpha_1}  & = \left\{  d
\hat{g}_{\alpha_0\alpha_1}  \right\}^{0,1} +\delta( \bar\partial
f^0_{\alpha_0} ) =
\left\{ d \theta^0_{\alpha_0\alpha_1} + \tau (d^B
t^0_{\alpha_0\alpha_1})  \right\}^{0,1} +\bar\partial f^0_{\alpha_0}
\\
& = \left\{ d \theta^0_{\alpha_0\alpha_1} + \tau (
h^1_{\alpha_0\alpha_1} ) - \delta(t^1_{\alpha_0})  \right\}^{0,1}
+\delta( \bar\partial f^0_{\alpha_0}) \\
& = \left\{ - \delta (\theta^1_{\alpha_0}) - (\tau
h^1_{\alpha_0\alpha_1}) + (\tau h^1_{\alpha_0\alpha_1}) -
\delta(t^1_{\alpha_0})  \right\}^{0,1} +\delta( \bar\partial
f^0_{\alpha_0} ) \\
& = - \delta \left(  \theta^1_{\alpha_0} + \tau t^1_{\alpha_0}
\right)^{0,1} +\delta( \bar\partial f^0_{\alpha_0} ) \\
& = -\delta ( a^{0,1}_{\alpha_0} )  +\delta( \bar\partial
f^0_{\alpha_0} ) = 0,
\end{align*}
cf. \eqref{eq:dbarsol}.

Defining
\begin{equation}
\label{eq:gab}
g_{\alpha_0\alpha_1} := \exp\left( \frac{1}{2\pi i} \tilde{g}_{\alpha_0\alpha_1} \right)
\end{equation}
one obtains a holomorphic structure $(g_{\alpha_0\alpha_1})$ for
$L$. See Remark \ref{rem:cocycle1}(i).

Now, define
\begin{equation}
\label{eq:varpsi}
\psi_{\alpha_0} := \frac{1}{2\pi i}\left( \partial f_{\alpha_0} - a^{1,0}_{\alpha_0}  \right) = \frac{1}{2\pi i} \left( d f_{\alpha_0} - \theta^1_{\alpha_0} - \tau t^1_{\alpha_0} \right).
\end{equation}
It follows from \eqref{eq:dbarsol} and \eqref{eq:dbar2} that $
\psi_{\alpha_0}$ is a holomorphic $1$-form. Furthermore,
\begin{align*}
\left( \frac{dg_{\alpha_0\alpha_1}}{g_{\alpha_0\alpha_1}} \right)  & =
\frac{1}{2\pi i} \left\{ (d \hat{g}_{\alpha_0\alpha_1}) + \delta( d
f_{\alpha_0} ) \right\} \\
 & =
\frac{1}{2\pi i} \left\{ ( d \theta^0_{\alpha_0\alpha_1}) + (\tau d^B
t^0_{\alpha_0\alpha_1}) + \delta ( \partial f_{\alpha_0}) + \delta (
\bar \partial f_{\alpha_0})   \right\} \\
&=
\frac{1}{2\pi i} \left\{  - \delta(\theta^1_{\alpha_0}) -( \tau
h^1_{\alpha_0\alpha_1}) + \tau\left( (h^1_{\alpha_0\alpha_1}) -
\delta(t^1_{\alpha_0}) \right) \right. \\
& \ \ \left.  + \delta(\partial f_{\alpha_0}) +
\delta ( \bar\partial f_{\alpha_0} )  \right\} \\
&=
\frac{1}{2\pi i} \left\{ - \delta ( \underbrace{(\theta^1_{\alpha_0})
  + (\tau t^1_{\alpha_0})}_{a_{\alpha_0}} )  + \delta(\partial
f_{\alpha_0}) + \delta ( a^{0,1}_{\alpha_0} )   \right\} \\
&=
\frac{1}{2\pi i} \left\{ \delta( \partial f_{\alpha_0} -
a^{1,0}_{\alpha_0} )   \right\} \\
& =
\delta ( \psi_{\alpha_0}).
\end{align*}
See Remark \ref{rem:cocycle1}(ii).

\begin{remark}
\label{rem:omega}
It follows from \eqref{eq:varpsi}, \eqref{eq:t1I} and \eqref{eq:triple} that
\begin{align*}
d\psi_{\alpha_0} & = - \frac{1}{2\pi i} \left( d \theta^1_{\alpha_0} + d \tau(t^1_{\alpha_0})  \right)
= - \frac{1}{2\pi i} \left( d \theta^1_{\alpha_0} +  \tau(d^B t^1_{\alpha_0})  \right) \\
& = - \frac{1}{2\pi i} \left( d \theta^1_{\alpha_0} +  \tau(h^2_{\alpha_0})  \right)
= - \frac{1}{2\pi i} \omega^2_{\alpha_0}.
\end{align*}
\end{remark}

Therefore, $(\psi_{\alpha_0})$ defines a holomorphic connection
$\nabla$ on the holomorphic line bundle $L$ associated to
$(g_{\alpha_0\alpha_1})$. Since both $\hat{g}_{\alpha_0\alpha_1}$ and
$f_{\alpha_0}$ are equivariant functions, it follows that
$$
\overline{\sigma^* g_{\alpha_0\alpha_1}} \cdot g_{\alpha_0\alpha_1} = 1,
$$
and this defines a holomorphic isomorphism $\mbq \colon L \otimes
\overline{\sigma^* L} \to \bone$
which becomes a positive definite hermitian form on
$X(\bbr)$. (See Remark \ref{rem:cocycle1}(iii).) Finally, the identity
$\overline{\sigma^* \psi_{\alpha_0}} + \psi_{\alpha_0} = 0$ shows that
$\mbq$ is parallel with respect to the connection on $L \otimes
\overline{\sigma^* L}$ induced by $\nabla$.
See Remark \ref{rem:cocycle1}(iv).

We have thus associated to $\mbc$, with $[\mbc] \in \ker{\Psi}$, a
triple $(L,\mbq,\nabla)$ of elements satisfying the conditions in
Definition \ref{def:rlbc}. This gives a well-defined homomorphism
\begin{equation}
\label{eq:Phi}
\Phi \colon  \ker{\Psi} \longrightarrow \rlbc{X}.
\end{equation}
We now proceed to show that this is in fact an isomorphism.
\medskip

\begin{lemma}
\label{lem:rem}
There is a natural transformation $\Theta \colon \rlbc{-} \to
\bcz{2}{2}{-}$ such that for all $U \in \gman$ the following diagram
commutes
$$
\xymatrix{
\ker{\Psi_U}  \ar[r]^-{i}  \ar[d]_-{\Phi} &  \dcr{2}{2}{U} \ar[d]^{\varrho} \\
\rlbc{U} \ar[r]_{\Theta_U} & \bcz{2}{2}{U}.
}
$$
\end{lemma}
\begin{proof}
Using the interpretation of $\bcz{2}{2}{U}$ as equivalence classes
$[L,\mbq]$ of pairs as in Proposition \ref{prop:h22br}, then
$\Theta_U$ simply sends $\la L,\nabla, \mbq \ra$ to $[L,\mbq]$.
\end{proof}
\smallskip

\noindent{\bf Injectivity of $\Phi$:}\ Let $\mbc' = (\mbh', \underline{\omega'},
\underline{\theta'})$ be a $\check{\text{C}}$ech cocyle representing a
class $[\mbc'] \in \ker{\Psi_U}$, and suppose that $\Phi_U([\mbc']) =
0$. Since $0=\Theta_U\circ \Phi_U ([\mbc']) = \varrho \circ i ([\mbc])$,
one concludes that
$[\mbh'] = 0 \in \bcz{2}{2}{U}$ and hence, one can find $\mbt \in
\operatorname{Tot}(\check{C}^*(\underline{\calu},\br{\bbz}{2}))^1$
such that $\check{d}^B\mbt = \mbh$. It follows that $\mbc := \mbc' -
\check{D}(\mbt,0,0) = (0,\underline{\omega},\underline{\theta})$
represents $[\mbc']$ and is simply given by
$(\omega_{\alpha_0}^2), (\theta^0_{\alpha_0\alpha_1})$ and
$(\theta_{\alpha_0}^1)$, where
$\omega_{\alpha_0}^2 \in F^2\cale^2(U)$, $\theta^0_{\alpha_0\alpha_1}
\in \cale^0(U_{\alpha_0\alpha_1})$ and $\theta^1_{\alpha_0} \in
\cale^1(U_{\alpha_0})$.

Let $(g_{\alpha_0\alpha_1}), (\psi_{\alpha_0})$ be constructed as
in \eqref{eq:gab} and \eqref{eq:varpsi}, representing
$\Phi([\mbc])$, and observe that  $\mbh = 0$ allows one to take
$t^1_{\alpha_0}=0$, $t^0_{\alpha_0\alpha_1}=0$ and
$f^0_{\alpha_0\alpha_1}=0$ (cf. STEPS 1 and 2, and \eqref{eq:nu})
in their definition. Assuming that $\Phi([\mbc])=0$, one can find
$(\rho_{\alpha_0})$ satisfying:
\begin{align}
\rho_{\alpha_0}\cdot \overline{\sigma^*\rho_{\alpha_0}} & = 1, \quad
\text{and} \quad \rho_{\alpha_0} \text{ is holomorphic} \\
(g_{\alpha_0\alpha_1}) & = \delta(\rho_{\alpha_0}) \\
\psi_{\alpha_0} &= \frac{d\rho_{\alpha_0}}{\rho_{\alpha_0}}.
\end{align}
The latter equation gives $0 = d\psi_{\alpha_0} = -\frac{1}{2\pi
i} \omega^2_{\alpha_0}$; cf. Remark \ref{rem:omega}.

Now, the cocycle condition on $\mbc= (0,0,\underline{\theta})$ boils down to
\begin{align}
( d \theta^1_{\alpha_0}) & = 0 \\
(d\theta^0_{\alpha_0\alpha_1}) + \delta(\theta_{\alpha_0}^1) & = 0 \\
\delta(\theta^0_{\alpha_0\alpha_1}) & = 0.
\end{align}

Choose $\hrz $ such that $\exp{\hrz} = \rho_{\alpha_0}$ and $\hrz +
\overline{\sigma^* \hrz} =0$, and let $\fz$ be as in
\eqref{eq:dbarsol}.
Hence, by definition one has
$$
(g_{\alpha_0\alpha_1}) = \left( \exp \frac{1}{2\pi i} \left\{
(\theta^0_{\alpha_0\alpha_1} )+ \delta(\fz) ) \right\}   \right) =
\delta (\rho_{\alpha_0})
$$
and
\begin{equation}
\label{eq:hrz1}
\psi_{\alpha_0} = \frac{1}{2\pi i } \left( d\fz -
\theta^1_{\alpha_0} \right) =
\frac{d\rho_{\alpha_0}}{\rho_{\alpha_0}} = d \hrz.
\end{equation}
It follows that one can find $n_{\alpha_0\alpha_1} \in
\uZ(U_{\alpha_0\alpha_1})$ such that
\begin{equation}
\label{eq:hrz2}
\frac{1}{2\pi i} \left\{ (\theta^0_{\alpha_0\alpha_1} )+
\delta(\fz) ) \right\} = \delta(\hrz) + (2\pi i \
n_{\alpha_0\alpha_1}).
\end{equation}
Using Lemma \ref{lem:hom} define
\begin{equation}
\xi_{\alpha_0\alpha_1} := \xi_{(2\pi i)^2} (n_{\alpha_0 \alpha_1})
\in \br{\bbz}{2}^0(U_{\alpha_0\alpha_1})
\end{equation}
and define
\begin{equation}
\label{eq:b0}
b^0_{\alpha_0} = 2\pi i \hrz - \fz \in \cale^0(U_{\alpha_0}).
\end{equation}
It is easy to see that $\delta(\xi_{\alpha_0\alpha_1}) = 0$ and
that \eqref{eq:hrz2} and \eqref{eq:b0} give
\begin{equation}
\label{eq:coboundary0}
\tau(\xi_{\alpha_0\alpha_1}) + \delta(b^0_{\alpha_0}) =
(\theta^0_{\alpha_0\alpha_1}).
\end{equation}
One obtains a $1$-cochain $\mbt := (
(\xi_{\alpha_0\alpha_1} ),\ 0,\ ( b^0_{\alpha_0} ) ) \in
\operatorname{Tot}\left(
\check{C}^*(\underline{\calu},\br{\bbz}{2} ) \right)^1$ that
satisfies $ \check{D} (\mbt) = \mbc; $ cf. \eqref{eq:hrz1},
\eqref{eq:b0}, \eqref{eq:coboundary0} and \eqref{eq:varpsi}.
Therefore $[\mbc] = 0$, thus showing the
injectivity of $\Phi$.
\medskip

\noindent{\bf Surjectivity of $\Phi$:}\
\smallskip

Represent $\la L, \nabla, \mbq \ra \in \rlbc{X}$  by a cocycle
$(G_{\alpha_0\alpha_1}), (\psi_{\alpha_0})$ satisfying the
conditions of Remark \ref{rem:cocycle1}, and let $\mbh \in
\operatorname{Tot}\left(
\check{C}^*(\underline{\calu};\br{\bbz}{2} )\right)^2$ be a
cocycle representing the element $\Theta ( \la L, \nabla, \mbq \ra
) = [L,\mbq] \in \bcz{2}{2}{X}$. Note that this cocycle is
``unobstructed'', in the sense of {\sc Type $\frS$ case}. Hence one can
find $t^1_{\alpha_0}, t^0_{\alpha_0\alpha_1}$ and
$f^0_{\alpha_0\alpha_1}$ as in \eqref{eq:t1I}, \eqref{eq:t0I} and
\eqref{eq:nu}, respectively. By definition,
$\hat{g}_{\alpha_0\alpha_1} := f^0_{\alpha_0\alpha_1} +
\tau(t^0_{\alpha_0\alpha_1})$ so that $g_{\alpha_0\alpha_1} :=
\exp\left( \frac{1}{2\pi i} \hat{g}_{\alpha_0\alpha_1}\right)$ is
a cocycle representing the isomorphism class of $L$ as a smooth
line bundle and satisfying the condition
$g_{\alpha_0\alpha_1}\cdot \overline{\sigma^*g_{\alpha_0\alpha_1}}
=1$, which gives the desired $\mbq$, positive definite over
$X(\bbr)$.

Therefore, one can find smooth functions $\rho_{\alpha_0}$ such that
\begin{equation}
\label{eq:local}
\left( \frac{G_{\alpha_0\alpha_1}}{g_{\alpha_0\alpha_1}} \right) =
\delta( \rho_{\alpha_0}) \quad \text{and} \quad \rho_{\alpha_0}\cdot
\overline{\sigma^* \rho_{\alpha_0}} =1.
\end{equation}

Now, find $\hGG$ and $\hrz$ such that $\exp{\hGG} =
G_{\alpha_0\alpha_1},\ \exp\hrz = \rho_{\alpha_0}$ and satisfying
$\hGG + \overline{\sigma^* \hGG} =0, \ \hrz + \overline{\sigma^* \hrz}
= 0$.

It follows from \eqref{eq:local} that one can find
$n_{\alpha_0\alpha_1} \in \uZ(U_{\alpha_0\alpha_1})$ such that
$(\hGG) =  \frac{1}{2\pi i}\left\{ \hat{g}_{\alpha_0\alpha_1} + \delta
(\hrz) \right\} + (2\pi i) n_{\alpha_0\alpha_1}$. Hence:
\begin{equation}
\label{eq:n}
2\pi i \hGG = f^0_{\alpha_0\alpha_1} + \tau t^0_{\alpha_0\alpha_1} +
\delta(\hrz) + (2\pi i)^2 n_{\alpha_0\alpha_1}.
\end{equation}

Want to find $\omega_{\alpha_0}, \theta_{\alpha_0}^1,
\theta^0_{\alpha_0\alpha_1}$ satisfying:

\begin{align}
d\omega_{\alpha_0} & = 0 \label{eq:f1} \tag{C.1} \\
\delta (\omega_{\alpha_0}) & = 0 \label{eq:f2}\tag{C.2}\\
\delta ( \theta^0_{\alpha_0\alpha_1}) & = (\tau
h^0_{\alpha_0\alpha_1\alpha_2}) \label{eq:f3}\tag{C.3} \\
(d \theta^0_{\alpha_0\alpha_1}) + \delta (\theta^1_{\alpha_0}) + \tau
(h^1_{\alpha_0\alpha_1}) & = 0 \label{eq:f4}\tag{C.4} \\
\omega_{\alpha_0} &= d\theta^1_{\alpha_0} + \tau
h^2_{\alpha_0}. \label{eq:f5} \tag{C.5}
\end{align}
Note that \eqref{eq:f5} implies \eqref{eq:f1}.

It follows from \eqref{eq:n} that
\begin{align*}
2\pi i \delta(\psi_{\alpha_0}) & = 2 \pi i d \hGG  =
df^0_{\alpha_0\alpha_1} + \tau \left( d^Bt^0_{\alpha_0\alpha_1}
\right) + \delta(d \hrz) \\
& = df^0_{\alpha_0\alpha_1} + \tau \left( h^1_{\alpha_0\alpha_1} -
\delta( t^1_{\alpha_0}) \right) + \delta(d \hrz) \\
& = df^0_{\alpha_0\alpha_1} + \tau h^1_{\alpha_0\alpha_1} + \delta\left(
- \tau t^1_{\alpha_0} + d \hrz \right).
\end{align*}
Therefore,
\begin{equation}
\label{eq:therefore}
0 = ( df^0_{\alpha_0\alpha_1} ) + ( \tau h^1_{\alpha_0\alpha_1} ) +
\delta\left( -2\pi i \psi_{\alpha_0} - \tau t^1_{\alpha_0} + d \hrz
\right)
\end{equation}

Define
\begin{align}
\theta^1_{\alpha_0} & := -2\pi i \psi_{\alpha_0} - \tau t^1_{\alpha_0}
+ d \hrz  \label{eq:tt1} \\
\theta^0_{\alpha_0\alpha_1} & := f^0_{\alpha_0\alpha_1} \label{eq:tt2}\\
\omega_{\alpha_0} & := - 2 \pi i d \psi_{\alpha_0} \label{eq:tt3},
\end{align}
and observe that the latter is an invariant closed form of Hodge type
$(2,0)$, since $\psi_{\alpha_0}$ is holomorphic. We now proceed to
show that these forms satisfy \eqref{eq:f1}--\eqref{eq:f5}.

\begin{itemize}
\item $d\theta^1_{\alpha_0} = - 2\pi i d\psi_{\alpha_0} - \tau( d^B
t^1_{\alpha_0}) = \omega_{\alpha_0} - \tau(h^2_{\alpha_0})$. This
gives \eqref{eq:f5} and \eqref{eq:f1}, as well.
\item $\delta(\omega_{\alpha_0}) = - 2 \pi i \delta ( d
\psi_{\alpha_0} ) = - 2 \pi i d \delta(\psi_{\alpha_0}) = - 2 \pi i d
\left( \frac{d G_{\alpha_0\alpha_1}}{G_{\alpha_0\alpha_1}} \right) =
\\ - 2 \pi i d ( d \hGG )=0 $. This gives \eqref{eq:f2}.
\item $\delta( \theta^0_{\alpha_0\alpha_1}) = \delta(
f^0_{\alpha_0\alpha_1} ) = (\tau
h^0_{\alpha_0\alpha_1\alpha_2})$. This gives \eqref{eq:f3}.
\item It follows from \eqref{eq:therefore} and \eqref{eq:tt1} that $ 0
= (df^0_{\alpha_0\alpha_1}) + ( \tau h^1_{\alpha_0\alpha_1}) + \delta(
\theta^1_{\alpha_0\alpha_1} )$ which, together with \eqref{eq:tt2},
implies \eqref{eq:f4}.
\end{itemize}

It follows that $\mbc := ( \mbh, \underline{\omega},
\underline{\theta})$ gives a cocycle in $\operatorname{Tot}\left(
\check{C}^*(\underline{\calu};\de{\bbz}{2}) \right)^2$ such that
$[\mbc] \in \ker{\Psi}$.

Finally, one needs to verify that
$\Phi( [\mbc] ) = \la L,\nabla,\mbq \ra$. At this point, this is a
mere tautology. Following the steps in the definition of $\Phi$, one
constructs $(\gamma_{\alpha_0\alpha_1}), (\xi_{\alpha_0})$
representing $\Phi([\mbc])$.

We first find $f^0_{\alpha_0} \in \cale^0(U_{\alpha_0})$
such that $\bar{\partial} f^0_{\alpha_0} = \left\{
\theta^1_{\alpha_0} + \tau (t^1_{\alpha_0}) \right\}^{0,1}$; cf.
\eqref{eq:dbarsol}. Note that, by definition \eqref{eq:tt1}, one
has $ \theta^1_{\alpha_0} + \tau( t^1_{\alpha_0} ) = -2 \pi i
\psi_{\alpha_0} + d \hrz, $ and since $\Psi_{\alpha_0}$ has Hodge
type $(1,0)$, one concludes that $ \{ \theta^1_{\alpha_0} + \tau(
t^1_{\alpha_0} ) \}^{0,1} = \bar{\partial} \hrz. $ Hence, we can
choose $f^0_{\alpha_0} = \hrz$.

By definition,
\begin{align*}
\gamma_{\alpha_0\alpha_1} & = \exp\left( \frac{1}{2 \pi i}\left\{
\hat{g}_{\alpha_0\alpha_1} + \delta(\hrz) \right\}  \right) \\
& = \exp\left( \frac{1}{2 \pi i}\left\{  \hat{g}_{\alpha_0\alpha_1} +
\delta(\hrz) \right\} +  ( 2 \pi i n_{\alpha_0\alpha_1}) \right) \\
& = \exp( \hGG ) = G_{\alpha_0\alpha_1}.
\end{align*}

Also,
\begin{align*}
\xi_{\alpha_0} & = \frac{1}{2 \pi i}\left(  \partial
\hat{\rho}_{\alpha_0} - \{ \theta^1_{\alpha_0} + \tau t^1_{\alpha_0}
\}^{1,0} \right)
=
\frac{1}{2 \pi i}\left(  d \hat{\rho}_{\alpha_0} - \{
\theta^1_{\alpha_0} + \tau t^1_{\alpha_0} \} \right) \\
& =
\frac{1}{2 \pi i}\left(  2 \pi i \psi_{\alpha_0} \right) =
\psi_{\alpha_0}.
\end{align*}
\end{proof}

\bibliographystyle{amsalpha}
\bibliography{references}

\end{document}